\documentclass{amsart}
\usepackage[utf8]{inputenc}
\usepackage{geometry}
\usepackage{amstext}
\usepackage{amssymb}
\usepackage{color}

\makeatletter

\let\SF@@footnote\footnote
\def\footnote{\ifx\protect\@typeset@protect
    \expandafter\SF@@footnote
  \else
    \expandafter\SF@gobble@opt
  \fi
}
\expandafter\def\csname SF@gobble@opt \endcsname{\@ifnextchar[
  \SF@gobble@twobracket
  \@gobble
}
\edef\SF@gobble@opt{\noexpand\protect
  \expandafter\noexpand\csname SF@gobble@opt \endcsname}
\def\SF@gobble@twobracket[#1]#2{}

\usepackage{amsthm}
\usepackage{enumitem}		
 \newcommand\thmsname{Theorem}
 \newcommand\nm@thmtype{thm}
 \theoremstyle{plain}
 
 \newenvironment{namedthm}[1]{
   \renewcommand\thmsname{#1}\renewcommand\nm@thmtype{namedtheorem}
   \begin{\nm@thmtype}
}
   {\end{\nm@thmtype}
}

\usepackage{rotating}

\usepackage{amsfonts}\usepackage{amscd}\usepackage[all]{xy}\usepackage{color}\usepackage{tikz}\usepackage{amsthm}

\newtheorem{thm}{Theorem}[section]
\newtheorem*{thm*}{Theorem}

\newtheorem*{cor*}{Corollary}
\newtheorem{lem}[thm]{Lemma}
\newtheorem*{lem*}{Lemma}
\newtheorem{prop}[thm]{Proposition}
\newtheorem*{prop*}{Proposition}
\newtheorem{conjecture}[thm]{Conjecture}
\newtheorem*{conjecture*}{Conjecture}

\newtheorem*{fact*}{Conjecture}

\newtheorem*{criterion*}{Criterion}

\newtheorem*{algorithm*}{Algorithm}

\newtheorem*{ax*}{Axiom}

\newtheorem*{assumption*}{Assumption}
\newtheorem{question}[thm]{Question}
\newtheorem*{question*}{Question}

\theoremstyle{remark}

\newtheorem{rem}[thm]{Remark}
\newtheorem*{rem*}{Remark}

\newtheorem*{rems*}{Remarks}

\newtheorem*{claim*}{Claim}

\newtheorem*{exercise*}{Exercise}

\newtheorem*{note*}{Note}
\newtheorem{notation}[thm]{Notation}
\newtheorem*{notation*}{Notation}

\newtheorem*{summary*}{Summary}

\newtheorem*{acknowledgement*}{Acknowledgement}

\newtheorem*{conclusion*}{Conclusion}

\theoremstyle{definition}

\newtheorem{defn}[thm]{Definition}
\newtheorem*{defn*}{Definition}

\newtheorem*{example*}{Example}

\newtheorem*{examples*}{Examples}

\newtheorem*{problem*}{Problem}

\newtheorem*{xca*}{Exercise}

\newtheorem*{xcas*}{Exercises}

\newtheorem*{condition*}{Condition}

\def\rest{\hspace{-0.6ex}\upharpoonright\hspace{-0.6ex}}
\def\restr{\hspace{-0.6ex}\upharpoonright}
\def\restri{\upharpoonright}

\renewcommand{\bar}{\overline}

\author[Jones]{Gareth Jones}
\address{School of Mathematics,
The University of Manchester,
Oxford Rd,
Manchester,
M13 9PL,
UK}
\email{Gareth.Jones-3@manchester.ac.uk}
\urladdr{http://personalpages.manchester.ac.uk/staff/Gareth.Jones-3/index.php}

\author[Kirby]{Jonathan Kirby}
\address{School of Mathematics, University of East Anglia, Norwich Research Park, Norwich, NR4 7TJ, UK}
\email{jonathan.kirby@uea.ac.uk}
\urladdr{http://www.uea.ac.uk/mathematics/people/profile/jonathan-kirby}

\author[Le Gal]{Olivier Le Gal}
\address{Univ. Grenoble Alpes, Univ. Savoie Mont Blanc, CNRS,
LAMA, 73000 Chamb\'ery, France}
\email{Olivier.Le-Gal@univ-smb.fr}
\urladdr{http://www.lama.univ-savoie.fr/~LeGal/}

\author[Servi]{Tamara Servi}
\address{Institut de Math\'ematiques de Jussieu - Paris Rive Gauche,
Universit\'e Paris Diderot (Paris 7),
UFR de Math\'ematiques - B\^atiment Sophie Germain,
Case 7012,
F-75205 Paris Cedex 13,
France}
\email{tamara.servi@math.univ-paris-diderot.fr}
\urladdr{http://www.logique.jussieu.fr/~servi/index.html}



\makeatother

\begin{document}

\title{On local definability of holomorphic functions}

\begin{abstract}
	Given a collection $\mathcal{A}$ of holomorphic functions, we consider how to describe all the holomorphic functions locally definable from $\mathcal{A}$. The notion of local definability of holomorphic functions was introduced by Wilkie, who gave a complete description of all functions locally definable from $\mathcal{A}$ in the neighbourhood of a generic point. We prove that this description is no longer complete in the neighbourhood of non-generic points. More precisely, we produce three examples of holomorphic functions which suggest that at least three new operations need to be added to Wilkie's description in order to capture local definability in its entirety. The constructions illustrate the interaction between resolution of singularities and definability in the o-minimal setting.\end{abstract}

\maketitle
\begin{center}
\today
\par\end{center}

\begin{acknowledgement*}
The first author was partially suppoerted by the EPSRC (EP/J01933X/1 and EP/N007956/1).
The second author was partially supported by the EPSRC (EP/L006375/1).
The third author is partially supported by STAAVF (ANR-11-BS01-0009).
The fourth author is partially supported by ValCoMo (ANR-13-BS01-0006).
\end{acknowledgement*}

\section{Introduction\label{sec:Introduction}}

There has long been interaction between the theory of first-order definability and functional transcendence. For instance \cite{bianconi:sinus_exponential,jks,Sfouli,legal_rolin:Ck_fourier,lg:generic,serge} use various results on algebraic independence of certain functions to establish results on definability and nondefinability in o-minimal expansions of the real field. The nondefinability results give in turn natural strengthenings of statements from functional transcendence. For instance, van den Dries, Macintyre and Marker, showed, among many other things, that $\int \exp(-x^2)\,dx$ is not definable in $\mathbb R_{\text{an, exp}}$ (Theorem 5.11 in \cite{vdd_mac_mar}). This gives a much stronger form of Liouville's result that this antiderivative is not an elementary function (\cite{liouville}). 

In this spirit, we investigate the holomorphic functions which are locally definable from a given collection $\mathcal{A}$ of holomorphic functions. Roughly speaking, these are the functions whose graph (seen as a subset of $\mathbb{R}^{2N}$, for a suitable $N$) can be described, locally, by a first order formula involving finite sums, products and compositions of (the real and imaginary parts of) functions in $\mathcal{A}$ (see Definition \ref{def:locally def}). It is natural, both from the logical point of view and also from the perspective of analytic geometry, to seek a complex analytic characterization of the holomorphic functions locally definable from $\mathcal{A} $.

Wilkie gave such a characterization around \emph{generic} points in \cite{wil:conj}. He showed that if $f$ is a locally definable holomorphic function then around a generic point $f$ is contained in the smallest collection of holomorphic functions containing both the functions in $\mathcal{A}$ and all polynomials, and closed under partial differentiation, Schwarz reflection, composition and extraction of implicit functions. Wilkie conjectured that the same result would hold even without the genericity hypothesis. {The aim of this paper is to} show that in fact this description is no longer complete around non-generic points. We give three examples of functions which between them suggest that at least three further operations need to be added to those considered by Wilkie in order to capture local definability. We construct functions $f_1,f_2,f_3$ and $g_1,g_2,g_3$ with the germ of $f_i$ definable from $g_i$, but not obtainable from $g_i$ by the operations above. To obtain $f_i$ from $g_i$ requires further operations: $f_1$ needs monomial division, $f_2$ needs deramification and $f_3$ needs blowing-down (Theorems A, B and C, respectively, in Section \ref{sec:Notations,-operators-and}).

These three new operations arise naturally in resolution of singularities \cite{bm_semi_subanalytic}, which is an essential tool from analytic geometry used in o-minimality (see for example \cite{rsw,rol:ser:eq}). This is consistent with Wilkie's observation that his conjecture is related to resolution of singularities. It is natural to ask whether the operations introduced so far would suffice to give a complete description of the holomorphic functions locally definable from $\mathcal{A}$ (see Question \ref{conj+}). Indeed exactly these operations appear in the piecewise description of definable functions given in \cite{rol:ser:eq}. And based on the latter, it will be shown in the forthcoming \cite{lgsvb} that, if $\mathcal{A}$ is a collection of \emph{real} analytic functions, then the real analytic functions which are locally definable from $\mathcal{A}$ 
can be obtained from $\mathcal{A}$ and polynomials by the same list of operations (except Schwarz reflection). 
Since in the description given in \cite{rol:ser:eq},
the operations involved do not respect 
{any} underlying complex structure
(for example, exceptional divisors of blow-downs generally have real
co-dimension $1$), we cannot directly deduce an answer to Question
\ref{conj+} from an answer to its real counterpart.


\medskip{}

We now briefly discuss the proofs of Theorems A, B and C.

For Theorem A, we use Ax's functional version of Schanuel's conjecture to prove that the function $(e^z-1)/z$ cannot be obtained from $e^z$ without monomial division. There is a connection here with previous work on nondefinability, for example \cite{bianconi:sinus_exponential,jks}. However, these papers concerned arbitrary first-order definability, reduced to existential definability via model completeness. Here, we consider a restricted form of definability characterized by analytic rather than logical considerations. 

{The proofs of Theorems B and C are less explicit. For Theorem B, we first observe that deramification is independent from the operations used by Wilkie and monomial division (see Proposition \ref{shift:prop}). It is however not easy to witness this independence with a natural function. So instead we build on some of the ideas contained in \cite{lg:generic}. In particular, we adapt the notion of strongly transcendental function considered there to the setting of several complex variables. A related notion of strongly transcendental function was also independently introduced by Boris Zilber \cite{Boris} under the name of `generic functions with derivatives', in connection with finding analogues of Schanuel's conjecture for other complex functions. Our present work then proves the existence of such functions, which was not considered by Zilber.}

{For Theorem C, the independence of blow-downs from the other operations simply follows from the fact that, unlike the other operators considered here, blow-downs are not local operators: they are applied to all the germs of a function along the exceptional divisor of the corresponding blow-up. As in the case of theorem B, it is not clear how to find an explicit function which witnesses this independence, and we again make use of strongly transcendental functions.}

\medskip{}
The paper is organized as follows. In Section \ref{sec:Notations,-operators-and}
we introduce some operators and algebras which allow us to give the
precise formulations of our results. In Section \ref{sec:Strong-Transcendence}
we define a notion of strong transcendence for holomorphic functions
in $\mathbb{C}^{2}$, and prove that the notion is not void and implies
an independence property which is the key point for constructing our
last two examples. Finally, in Sections \ref{sec:Monomial-division},
\ref{sec:Composition-with-roots} and \ref{sec:Blow-downs} we prove
Theorems A, B and C, respectively.


\section{Main definitions and results%
\label{sec:Notations,-operators-and}}

{In this section we introduce some notation, and the various operators and functions we consider. We then give precise statements of our results.}

Throughout this paper, we will use the word \emph{definable} in the
sense of first-order logic. Unless otherwise specified, sets and functions
definable in a given first-order structure will be understood to be
definable with parameters from the domain of the structure. We will
use the term \emph{$\emptyset$-definable} to denote sets and functions
definable without parameters.


Following Wilkie in \cite{wil:conj}, we say that the restriction $g\restriction\Delta$ of a holomorphic map $g:U\rightarrow\mathbb{C}$ defined on an open set $U\subseteq\mathbb{C}^{n}$ is a \textbf{proper restriction }if $\Delta\subset U$ is an open box, relatively compact in $U$. 
Given a family $\mathcal{A}$ of holomorphic functions defined on open subsets of $\mathbb{C}^{n}$ (for various $n\in\mathbb{N}$),
	we denote by $\mathcal{A}\restr$ the collection of all proper restrictions
	to boxes with corners in the Gaussian rationals, $\mathbb Q(i)$, of all functions in $\mathcal{A}$. We let $\mathbb{R}_{\mathcal{A}\restri}=\langle\mathbb{R}; 0,1,+,\cdot,<,\mathcal{A}\restr\,\rangle$ be the expansion of the real ordered field by the functions in $\mathcal{A}\restr$, seen as functions from some (even) power of $\mathbb R$ to $\mathbb R^2$, by identifying $\mathbb C$ with $\mathbb R^2$. 
	
	\begin{defn}\label{def:locally def}
A function $g:U\subset\mathbb C^n\rightarrow\mathbb{C}$ is \textbf{locally definable} in $\mathbb{R}_{\mathcal{A}\restri}$ if all its proper restrictions, interpreted as real functions, are definable in $\mathbb{R}_{\mathcal{A}\restri}$. If all the proper restrictions of $g$ to boxes with Gaussian rational corners are $\emptyset$-definable in $\mathcal{R}$, we say that $g$ is \textbf{locally $\emptyset$-definable}.
	\end{defn}


The symbol $\mathcal{A}$ will always denote
a collection of holomorphic functions defined on open subsets of $\mathbb{C}^{n}$,
for various $n\in\mathbb{N}$. More precisely, $\mathcal{A}$ is a
family $\left\{ \mathcal{A}_{n}\left(U\right):\ n\in\mathbb{N},\ U\subseteq\mathbb{C}^{n}\ \text{open}\right\} $,
where $\mathcal{A}_{n}(U)$ is a collection
of holomorphic functions defined on $U$.
It is convenient to suppose 
that $\mathcal{A}$
is closed under gluing and restrictions, which can always be done,
by enriching $\mathcal{A}$. More precisely, we suppose that,
if $(U_{\lambda})_{\lambda\in\Lambda}$ is a collection of open subsets
of $\mathbb{C}^{n}$, then
\[
f\in\mathcal{A}_{n}\left(\bigcup_{\lambda\in\Lambda}U_{\lambda}\right)\Leftrightarrow\forall\lambda\in\Lambda,f\rest U_{\lambda}\in\mathcal{A}_{n}(U_{\lambda}).
\]
Thus $\mathcal{A}_{n}$ can be viewed as a sheaf, for each fixed dimension
$n$, whose stalk at $z\in\mathbb C^n$, denoted by $\mathcal{A}_{z}$, is the collection of all germs at $z$ of
the functions which belong to some $\mathcal{A}_{n}(U)$ with $z\in U$. Note that closing $\mathcal{A}$
under gluing and restrictions does not affect local ($\emptyset$-)definability
in $\mathbb{R}_{\mathcal{A}\restri}$. By a classical argument of local compactness, writing $f_{z}$ for the germ
of $f$ at $z$, we have $f\in\mathcal{A}(U)$ if and only if for all $z\in U$, $f_{z}\in\mathcal{A}_{z}$.

By abuse of terminology, we will say that $\mathcal{A}$ is a sheaf,
and write $\mathcal{A}\subseteq\mathcal{B}$ as a shorthand for $\forall n\in\mathbb{N},\forall U\subseteq\mathbb{C}^{n},\ \mathcal{A}_{n}(U)\subseteq\mathcal{B}_{n}(U)$.
In the same way, if $\mathcal{A}_{n}$ is the sheafification 
 of a family of functions, we will often drop the index 
 $n$. For instance, we will write $\mathcal{A}=\{\exp\}$ as a shorthand
for
\[
\forall U\subseteq\mathbb{C}^{n},\;\mathcal{A}_{n}(U)=\begin{cases}
\emptyset & \text{ if }n\neq1\\
\{\exp_{\restri U}\} & \text{ if }n=1
\end{cases}.
\]


We fix a collection $\mathcal{A}$ of holomorphic functions as in
the discussion above. We now construct new sheaves $\mathcal{B},\mathcal{C},\mathcal{D},\mathcal{E},\mathcal{F}$
by closing $\mathcal{A}$ under certain operations. The operations
involved in the definition of $\mathcal{B},\mathcal{C},\mathcal{D}$
are local, hence we define them by their action on germs. We  denote by $\mathcal{O}_{n}(U)$ the collection
of all holomorphic functions defined on $U\subseteq\mathbb{C}^{n}$
and by $\mathcal{O}_{a}$ the collection of all holomorphic germs
at $a\in\mathbb{C}^{n}$. We let $\mathcal{P}$ be the collection of all complex polynomials
	(in any number of variables) and $\mathcal{P}_{G}$ be the subcollection
	of all polynomials with Gaussian rational coefficients. 
	We also use the following notation: $z'=\left(z_{1},\ldots,z_{n-1}\right)$
and $z=\left(z',z_{n}\right)$. If $a\in\mathbb{C}^{n}$ then we write
$a=\left(a',a_{n}\right)\in\mathbb{C}^{n-1}\times\mathbb{C}$.

\begin{defn}
\label{def:operators}
We define the following operators:
\begin{enumerate}
\item \textbf{Polynomial} and \textbf{Gaussian polynomial} operators.
The polynomial operators are operators of arity $0$, defined for each $P\in\mathcal{P}$ and $a\in\mathbb{C}^{n}$ which map the empty set to the germ $P_{a}\in\mathcal{O}_{a}$. The Gaussian polynomial operators are similar, but restricted to polynomials $P\in\mathcal{P}_G$.
\item The \textbf{Schwarz reflection} operator at $a\in\mathbb{C}^{n}$
maps a germ $f\in\mathcal{O}_{a}$ to the germ $z\mapsto\overline{f(\overline{z})}\in\mathcal{O}_{\overline{a}}$.
\item \textbf{Composition} operators: if $a\in\mathbb{C}^{n}\ \text{and\ }b\in\mathbb{C}^{m}$,
then we consider the operator that maps $(f,g_{1},\dots,g_{n})\in\mathcal{O}_{a}\times(\mathcal{O}_{b})^{n}$
to the germ $f\circ(g_{1},\dots,g_{n})\in\mathcal{O}_{b}$, whenever
$(g_{1}(b),\dots,g_{n}(b))=a$.
\item \textbf{Partial derivative} operators: for $a\in\mathbb{C}^{n}$ and
$j\in\{1,\dots,n\}$, the $j$-th partial derivative is the operator
that maps $f\in\mathcal{O}_{a}$ to the germ $\frac{\partial f}{\partial z_{j}}\in\mathcal{O}_{a}$.
\item The \textbf{implicit function} operator at $a\in\mathbb{C}^{n}$ maps
a germ $f\in\mathcal{O}_{a}$ satisfying $f(a)=0$ and $\frac{\partial f}{\partial z_{n}}(a)\neq0$
to the (unique) germ $\varphi\in\mathcal{O}_{a'}$ satisfying $f(z',\varphi(z'))\equiv0$.
\item \textbf{Monomial division} operators: if $a\in\mathbb{C}^{n}$, then
we consider the operator that maps $f\in\mathcal{O}_{a}$ to the germ
at $a$ of the extension by continuity of $z\mapsto\frac{f\left(z\right)}{(z_{n}-a_{n})}$,
whenever the germ $z'\mapsto f\left(z',a_{n}\right)\in\mathcal{O}_{a'}$
is identically zero.
\item \textbf{Deramification}%
: if $a\in\mathbb{C}^{n}$ and $m\in\mathbb{N}^{\times}$, then we
consider the operator which maps $f\in\mathcal{O}_{a}$ to the germ
$z\mapsto f(z',a_{n}+\sqrt[m]{z_{n}-a_{n}})\in\mathcal{O}_{a}$, whenever
the germ $f$ satisfies
\[
f\left(z\right)=f\left(z',a_{n}+e^{\frac{2\text{i}\pi}{m}}\left(z_{n}-a_{n}\right)\right).
\]
Here the condition on $f$ implies that this expression does
not depend on the choice of an $m^{\mathrm{th}}$-root for $z_{n}-a_{n}$.
\end{enumerate}
\end{defn}
Notice that elementary operators can be composed as long as their
arities match.
\begin{defn}
\label{def:sets of operators}We define $\mathcal{B}^{*}$
 to be the set of all operators which can be expressed
as finite compositions of polynomial,
Schwarz reflection, composition, partial derivative and implicit function
operators. The collection $\mathcal{B}^{\emptyset*}$ is defined similarly, with Gaussian polynomial operators in place of polynomial operators. Analogously, we let $\mathcal{C}^{*}$  and $\mathcal{C}^{\emptyset*}$ be the sets of all operators which can be expressed as
finite compositions of monomial division operators and operators in $\mathcal{B}^{*}$ and $\mathcal{B}^{\emptyset*}$,
respectively. And we let $\mathcal{D}^{*}$ and 
$\mathcal{D}^{\emptyset*}$ be the sets of all operators
which can be expressed as finite compositions of  deramification
operators and operators in $\mathcal{C}^{*}$
$\mathcal{C}^{\emptyset*}$, respectively.
\end{defn}
Once we have defined these operators, we can construct the set of
all functions which are “locally obtained from $\mathcal{A}$” by
the action of such operators. For example, $\mathcal{B}$ will denote
the set of all functions which are locally obtained from $\mathcal{A}\cup\mathcal{P}$
by composition, Schwarz reflection, taking partial derivatives and
extracting implicit functions.

More formally:
\begin{defn}
\label{def: algebras}
We let $\mathcal{B}, \mathcal{B}^{\emptyset}, \mathcal{C},\mathcal{C}^{\emptyset},\mathcal{D}$ and $\mathcal{D}^{\emptyset} $ be the smallest sheaves containing $\mathcal{A}$ with stalks stable under the action of the operators in $\mathcal{B}^*, \mathcal{B}^{\emptyset*}, \mathcal{C}^*,\mathcal{C}^{\emptyset*},\mathcal{D}^*$ and $\mathcal{D}^{\emptyset*} $, respectively.
\end{defn}
\begin{rem}
Note that if $U$ is an open neighbourhood of $0\in\mathbb{C}^{n}$
and $f\in\mathcal{A}\left(U\right)$ satisfies $f\left(x',0\right)=0$,
then for every $a\in U$ such that $a_{n}\not=$0, the germ $g_{a}$
at $a$ of $\frac{f\left(x',x_{n}\right)}{x_{n}}$ belongs to $\mathcal{B}_{a}$
(since $g_{a}$ is the implicit function of $x_{n+1}x_{n}-f\left(x',x_{n}\right)=0$
at $a$). A similar argument holds for functions obtained by taking
an $m^{\mathrm{th}}$-root. 
Hence, $\mathcal{C}$ and $\mathcal{D}$ are indeed the sheaves of functions locally obtained by taking implicit functions and by deramification. 
In particular, $\mathcal{B}^{\emptyset}$
coincides with the algebra denoted by $\tilde{\mathcal{F}}$ in \cite[Definition 1.7]{wil:conj}.
\end{rem}
It remains to define the functions locally obtained by blow-downs.
For this, we say that a sheaf $\mathcal{G}\subseteq\mathcal{O}$ is
stable under \textbf{blow-downs} if for every $n\in\mathbb{N}$, every open
subset $U\subseteq\mathbb{C}^{n}$ and every blow-up $\pi:V\to U$
of a smooth analytic manifold $X\subseteq U$, which is locally defined
by a system of equations in $\mathcal{G}$, we have that if $f\circ\pi\in\mathcal{G}(V)$
then $f\in\mathcal{G}(U)$. A more precise (but heavier) definition
could be given with no mention of $\mathcal{G}(V)$ (note that $V$
is not an open subset of a power of $\mathbb{C}$), using local coordinates,
but in this paper we will only mention the blow-up of $0\in\mathbb{C}^{2}$
and introduce the corresponding local coordinates in Section \ref{sec:Blow-downs}.
\begin{defn}
\label{def: E}Let $\mathcal{E}$  and $\mathcal{E}^{\emptyset}$
be the smallest sheaves which are stable under blow-downs and under the action of the operators in $\mathcal{D}^{*}$ and $\mathcal{D}^{\emptyset*}$,
respectively. 

Finally, we denote by $\mathcal{F}$ and $\mathcal{F}^{\emptyset}$ the sheaves of all locally definable holomorphic functions and all locally $\emptyset$-definable holomorphic functions, respectively. 
\end{defn}
Thanks to the compactness
of the fibres of blow-ups, if $f$ is the blow-down of $g$ and $g$
is locally definable in $\mathbb{R}_{\mathcal{A}\restri}$, then $f$
is also locally definable in $\mathbb{R}_{\mathcal{A}\restri}$.

\medskip 

By construction we have :
\[
\begin{array}{ccccccccccccc}
\mathcal{A} & \subseteq & \mathcal{B}^{\emptyset} & \subseteq & \mathcal{C}^{\emptyset} & \subseteq & \mathcal{D}^{\emptyset} & \subseteq & \mathcal{E}^{\emptyset} & \subseteq & \mathcal{F}^{\emptyset}\\
 &  & \mathbin{\rotatebox[origin=c]{-90}{\ensuremath{\subseteq}}} &  & \mathbin{\rotatebox[origin=c]{-90}{\ensuremath{\subseteq}}} &  & \mathbin{\rotatebox[origin=c]{-90}{\ensuremath{\subseteq}}} &  & \mathbin{\rotatebox[origin=c]{-90}{\ensuremath{\subseteq}}} &  & \mathbin{\rotatebox[origin=c]{-90}{\ensuremath{\subseteq}}}\\
 &  & \mathcal{B} & \subseteq & \mathcal{C} & \subseteq & \mathcal{D} & \subseteq & \mathcal{E} & \subseteq & \mathcal{F} & \subseteq & \mathcal{O}.
\end{array}\ \ \ (*)
\]

\medskip{}
We can now give a precise formulation of the results announced in
the introduction.

Wilkie proposed the following conjecture.
\begin{conjecture}
	[{\cite[Conjecture 1.8]{wil:conj}}]\label{conj} Let
	$\mathcal{A}$ be a family of complex holomorphic functions. Let $z\in\mathbb{C}^{n}$,
	$U$ be an open neighbourhood of $z$ and $f:U\rightarrow\mathbb{C}$
	be a holomorphic function which is locally $\emptyset$-definable
	in $\mathbb{R}_{\mathcal{A}\restri}$. Then there is an open box with
	Gaussian rational corners $\Delta\subseteq U$ such that $z\in\Delta$
	and $f\rest\Delta$ can be obtained from $\mathcal{A}\cup\mathcal{P}_{G}$
	by finitely many applications of composition, Schwarz reflection,
	taking partial derivatives and extracting implicit functions.
\end{conjecture}

Conjecture \ref{conj} can be restated as follows: $\mathcal{B}^{\emptyset}=\mathcal{F}^{\emptyset}$.
By adding all constant functions to $\mathcal{A}$, the conjecture
implies also that $\mathcal{B}=\mathcal{F}$. The conjecture is however not true in general. We show that 
the six horizontal inclusions between $\mathcal B^{\emptyset}$ and $\mathcal E$ in diagram ($*$) are strict in general:
\begin{namedthm}
	{Theorem A}[Monomial division]\label{cex:division} Suppose $\mathcal{A}=\{\exp\}$
	is the complex exponential function. Then the function $f: z\mapsto(e^{z}-1)/z$
	(extended by continuity at $z=0$) is holomorphic and locally
	$\emptyset$-definable in $\mathbb{R}_{\mathcal{A}\restri}$, but
	no restriction of $f$ to any neighbourhood of $0$ can be obtained
	from $\mathcal{A}\cup\mathcal{P}_{G}$ by finitely many applications
	of composition, Schwarz reflection, taking partial derivatives and
	extracting implicit functions. \\ More precisely, the germ at $0$ of $f$ belongs to $\mathcal{C}_{0}^{\emptyset}\setminus\mathcal{B}_{0}$.
	In particular, $\mathcal{B}^{\emptyset}\neq\mathcal{C}^{\emptyset}$
	and $\mathcal{B}\neq\mathcal{C}$.
	 
\end{namedthm}

\begin{namedthm}
	{Theorem B}[Deramification]\label{cex:deramif}
	There exists a holomorphic function $g$, with domain a neighbourhood
	of $0\in\mathbb{C}$, such that the function $f:z\mapsto g(\sqrt{z})$
	is well defined, holomorphic in a neighbourhood of $0$ and locally
	$\emptyset$-definable in $\mathbb{R}_{\mathcal{A}\restri}$ (where
	$\mathcal{A}=\{g\}$), but no restriction of $f$ to any neighbourhood
	of $0$ can be obtained from $\mathcal{A}\cup\mathcal{P}_{G}$ by
	finitely many applications of composition, Schwarz reflection, monomial
	division, taking partial derivatives and extracting implicit functions. \\ More precisely,   the germ at $0$ of $f$
belongs to $\mathcal{D}_{0}^{\emptyset}\setminus\mathcal{C}_{0}$.
In particular, $\mathcal{C}^{\emptyset}\neq\mathcal{D}^{\emptyset}$
and $\mathcal{C}\neq\mathcal{D}$.
\end{namedthm}

\begin{namedthm}
	{Theorem C}[Blow-down]\label{cex:effondrement} Let $\pi$ be the
	blow-up of the origin in $\mathbb{C}^{2}$. Then there exists a holomorphic
	function $f$, with domain a neighbourhood of $0\in\mathbb{C}^{2}$,
	such that, if $\mathcal{A}=\{f\circ\pi\}$, then $f$ is
	locally $\emptyset$-definable in $\mathbb{R}_{\mathcal{A}\restri}$,
	but no restriction of $f$ to any neighbourhood of $0$ can be obtained
	from $\mathcal{A}\cup\mathcal{P}_{G}$ by finitely many applications
	of composition, Schwarz reflection, monomial division, deramification, taking partial derivatives and extracting implicit
	functions. \\ More precisely, the germ at $0$
of $f$ belongs to $\mathcal{E}_{0}^{\emptyset}\setminus\mathcal{D}_{0}$.
In particular, $\mathcal{D}^{\emptyset}\neq\mathcal{E}^{\emptyset}$
and $\mathcal{D}\neq\mathcal{E}$. \end{namedthm}
\begin{rem}

Wilkie proves in \cite[Theorem 1.10]{wil:conj} that Conjecture \ref{conj}
holds for all points $z$ which are \emph{generic} with respect to
a suitable pregeometry associated to ${\mathcal{A}}$.
 Since the points at which we apply the operators in $\mathcal{E}^{*}\setminus\mathcal{B}^{*}$
are not generic, the theorems above are consistent with Wilkie's result. In particular, if $z$ is generic point, then we
can restate Wilkie's result \cite[Theorem 1.10]{wil:conj} in the following form: $\mathcal{B}_{z}^{\emptyset}=\mathcal{C}_{z}^{\emptyset}=\mathcal{D}_{z}^{\emptyset}=\mathcal{E}_{z}^{\emptyset}=\mathcal{F}_{z}^{\emptyset}$.
\end{rem}

The purpose of \cite[Conjecture 1.8 and Theorem 1.10]{wil:conj} was
to describe the holomorphic functions locally definable from $\mathcal{A}$
in terms of the functions in $\mathcal{A}$, using purely complex
operations. Since our examples are obtained by means of three natural
complex operations, we could consider a modification of conjecture \ref{conj}. But we prefer to formulate this as a question.
\begin{question}
	\label{conj+} Let $\mathcal{A}$ be a family of complex holomorphic
	functions. Let $z\in\mathbb{C}^{n}$, $U$ be an open neighbourhood
	of $z$ and $f:U\rightarrow\mathbb{C}$ be a holomorphic function
	which is locally $\emptyset$-definable in $\mathbb{R}_{\mathcal{A}\restri}$.
	Is there an open box with Gaussian rational corners $\Delta\subseteq U$
	such that $z\in\Delta$ and $f\rest\Delta$ can be obtained from $\mathcal{A}\cup\mathcal{P}_{G}$
	by finitely many applications of composition, Schwarz reflection,
	monomial division, deramification, blow-downs, taking
	partial derivatives and extracting implicit functions?
\end{question}
Question \ref{conj+} can therefore be restated as follows: do $\mathcal{E}^{\emptyset}$
and $\mathcal{F}^{\emptyset}$ coincide?

\medskip{}

We conclude this section with some considerations on the local operators defined above, which will be useful to prove our main results.

The operators in $\mathcal{D}^{*}$ act on germs of holomorphic functions,
hence they can also be seen as acting on the Taylor expansion of such
germs (we do not make a distinction between an analytic germ and
its Taylor expansion). Let us fix some notation.
\begin{notation}
\label{notation: tuples}We will often use the following notation
for tuples: if $k\in\mathbb{N}\setminus\{0\}$, then $z=z^{\left(k\right)}$ means
that $z$ is a $k$-tuple of variables. Similarly, we write $a=a^{\left(k\right)}$
for a point in $\mathbb{C}^{k}$.

If $\alpha\in\mathbb{N}^{k}$ and $z=\left(z_{1},\ldots,z_{k}\right)\in\mathbb{C}^{k}$
then $z^{\alpha}=z_{1}^{\alpha_{1}}z_{2}^{\alpha_{2}}\dots z_{k}^{\alpha_{k}}$.
We set $\alpha!=\alpha_{1}!\ldots\alpha_{k}!$, $|\alpha|=\sum_{i=1}^k\alpha_{i}$
and denote by $\partial_{\alpha}f$ the partial derivative $\frac{\partial^{|\alpha|}f}{\partial z^{\alpha}}$
of a function $f:\mathbb{C}^{k}\to\mathbb{C}$. If $\alpha=0$, then
$\partial_{\alpha}f=f$. \end{notation}
\begin{defn}
\label{def: jet}If $U\subseteq\mathbb{C}^{m}$, $f\in\mathcal{O}(U)$
and $a=(a_{1}^{\left(m\right)},\dots,a_{n}^{\left(m\right)})\in U^{n}$,
the \textbf{jet of order $k$ of $f$ at $a$ }is the tuple $j_{n}^{k}f(a)$
of all partial derivatives $\partial_{\alpha}f$ for $|\alpha|\le k$,
evaluated at the points $a_{i}$, for $i=1,\dots,n$:
\[
j_{n}^{k}f(a)=(\partial_{\alpha}f(a_{1}^{\left(m\right)}),\ldots,\partial_{\alpha}f(a_{n}^{\left(m\right)}))_{|\alpha|\le k}.
\]
 We omit the subscript $n$ if $n=1$.
\end{defn}
The action of our operators on Taylor expansions can be described
as follows.
\begin{prop}
\label{formalop} Let $k,n,m_{1},\ldots,m_{k}\in\mathbb{N}$ and $a_{1}\in\mathbb{C}^{m_{1}},\ldots,a_{k}\in\mathbb{C}^{m_{k}},b\in\mathbb{C}^{n}$.
Let $\mathcal{L}:\mathcal{O}_{a_{1}}\times\dots\times\mathcal{O}_{a_{k}}\to\mathcal{O}_{b}$
be an operator in $\mathcal{D}^{*}$, and $(f_{1},\dots,f_{k})\in\mathcal{O}_{a_{1}}\times\dots\times\mathcal{O}_{a_{k}}$
be a $k$-tuple of germs in the domain of definition of $\mathcal{L}$.
Then, there exist
\begin{itemize}
\item a neighbourhood $W$ of $(f_{1},\dots,f_{k})$ for the Krull topology,
\item a tuple of constants $c=(c_{1},\dots,c_{\ell})\in\mathbb{C}^{\ell}$
(for some $\ell\in\mathbb{N}$),
\item for each $\alpha\in\mathbb{N}^{n}$, an integer $n_{\alpha}\in\mathbb{N}$
and a polynomial $P_{\alpha}\in\mathbb{Q}[y,\overline{y}]$, where
$y=y^{\left(N_{\alpha}\right)}$ and $N_{\alpha}$ can be a computed
from $k,m_{1},\ldots,m_{k},\ell,n_{\alpha}$,%
\end{itemize}

such that, for each $(g_{1},\dots,g_{k})\in W$, $\mathcal{L}(g_{1},\dots,g_{k})$
has a Taylor expansion at $b$ of the form
\[
{\displaystyle \mathcal{L}(g_{1},\dots,g_{k})(x)=\sum_{\alpha\in\mathbb{N}^{n}}P_{\alpha}(c_{1},\ldots,c_{\ell},j^{n_{\alpha}}g_{1}(a_{1}),\dots,j^{n_{\alpha}}g_{k}(a_{k}))(x-b)^{\alpha}}.
\]

\end{prop}
\begin{proof}
We first prove that the proposition is true if $\mathcal{L}$ is an
elementary operator.

For all but the implicit function operator, the neighbourhood $W$
will be the whole space. Note that for all but the polynomial and
the composition operators, we have $k=1$ (hence in these cases we
will write $g$ instead of $g_{1}$).
\begin{enumerate}
\item If $\mathcal{L}$ is a polynomial operator, then $n_{\alpha}=0$ and
the constants $c_{i}$ can be computed from the coefficients of the
polynomial and the coordinates of $b$.
\item If $\mathcal{L}$ is the Schwarz reflection operator, then $n_{\alpha}=|\alpha|$
and $P_{\alpha}(j^{|\alpha|}g(a))=\overline{(\alpha!)^{-1}\partial_{\alpha}g(a)}$.
\item If $\mathcal{L}$ is a composition operator, then the proposition
follows from the so called Faà Di Bruno formula (see for example \cite[p. 92. §115]{arbogast:FaaDiBruno}),
with $n_{\alpha}=|\alpha|$.
\item If $\mathcal{L}$ is the partial derivative operator $\frac{\partial}{\partial z_{j}}$,
then $n_{\alpha}=|\alpha|+1$ and $P_{\alpha}(j^{n_{\alpha}}g(a))=\frac{\alpha_{j}}{\alpha!}\frac{\partial_{}}{\partial z_{j}}\left(\partial_{\alpha}g(a)\right)$.
\item If $\mathcal{L}$ is the implicit function operator, then set $c_{1}=\left(\frac{\partial f}{\partial z_{n}}(a)\right)^{-1}$,
$n_{\alpha}=|\alpha|$, and
\[
W=\left\{ g:\ \; g(a)=0,\;\frac{\partial g}{\partial z_{n}}(a)=\frac{\partial f}{\partial z_{n}}(a)\right\} .
\]
It is easy to see that the coefficients of the Taylor expansion up
to order $|\alpha|$ of $\mathcal{L}\left(g\right)$ at the point
$b=\left(a_{1},\ldots,a_{n-1}\right)$ can be expressed as polynomials
(with coefficients in $\mathbb{Z}\left[c_{1}\right]$) in the derivatives
up to order $\left|\alpha\right|$ of $g$ at the point $a$.
Hence, $n_{\alpha}=\left|\alpha\right|$ and $P_{\alpha}$ is a polynomial
in the jet of order $|\alpha|$ of $g$ and $c_{1}$. \\
In order for $P_{\alpha}$ to be a polynomial and not a rational function,
we need $\frac{\partial g}{\partial z_{n}}(a)$ to be constant and
hence to restrict our claim to the neighbourhood $W$.
\item If $\mathcal{L}$ is a monomial division operator, then $n_{\alpha}=|\alpha|+1$
and 
\[P_{\alpha}(j^{|\alpha|}g(a))= \left(\left(\alpha_{n}+1\right)\cdot\alpha!\right)^{-1} \frac{\partial}{\partial z_{n}}\left(\partial_{\alpha}g\left(a\right)\right).\]
\item If $\mathcal{L}$ is the $m$th-deramification operator,
then $n_{\alpha}=m|\alpha|$; we leave it to the reader to find the
expression for $P_{\alpha}$ in this case.
\end{enumerate}

Each operator in $\mathcal{D}^{*}$ is a composition of elementary
operators. Suppose that the proposition holds for an operator $\mathcal{L}$.
If $f\in\mathcal{O}_{b}$ is in the image of $\mathcal{L}$ and $V\subseteq\mathcal{O}_{b}$
is an open neighbourhood of $f$ in the Krull topology, then $\mathcal{L}^{-1}\left(V\right)$
is an open neighbourhood of $\mathcal{L}^{-1}\left(f\right)$. This
observation and an easy computation show that the conclusion of the proposition
is preserved under composition. We illustrate the proof for the composition
of two operators $\mathcal{L}:\mathcal{O}_{a}\rightarrow\mathcal{O}_{b}$
and $\mathcal{M}:\mathcal{O}_{b}\rightarrow\mathcal{O}_{d}$, the
general case following as a straightforward but tedious exercise.
Suppose that
\[
\begin{aligned}\mathcal{L}\left(f\right) & \left(y\right)= & \sum_{\alpha}P_{\alpha}\left(c_{\mathcal{L}},j^{n_{\alpha}}f\left(a\right)\right)\left(y-b\right)^{\alpha},\\
\mathcal{M}\left(g\right) & \left(z\right)= & \sum_{\beta}Q_{\beta}\left(c_{\mathcal{M}},j^{m_{\beta}}g\left(b\right)\right)\left(z-d\right)^{\beta}.
\end{aligned}
\]
Then $j^{m_{\beta}}\mathcal{L}\left(f\right)\left(b\right)=\left(\partial_{\gamma}\left(\sum_{\alpha}P_{\alpha}\left(c_{\mathcal{L}},j^{n_{\alpha}}f\left(a\right)\right)\left(y-b\right)^{\alpha}\right)\right)_{\left|\gamma\right|\leq m_{\beta}}=\left(\gamma!P_{\gamma}\left(c_{\mathcal{L}},j^{n_{\gamma}}f\left(a\right)\right)\right)_{\left|\gamma\right|\leq m_{\beta}}$.
Hence we can write
\[
\mathcal{M}\left(\mathcal{L}\left(f\right)\right)\left(z\right)=\sum_{\beta}\tilde{Q}_{\beta}\left(c_{\mathcal{L}},c_{\mathcal{M}},j^{m_{\beta}}f\left(a\right)\right)\left(z-d\right)^{\beta},
\]
for some suitable polynomials $\tilde{Q}_{\beta}$.


\end{proof}

\section{Strong Transcendence\label{sec:Strong-Transcendence}}

To prove Theorems B and C, we will need to work with holomorphic
functions which satisfy very few {functional} relations. For this reason, we extend
to the complex and multi-dimensional setting the notion of \emph{strong
transcendence}, which was introduced in \cite{lg:generic} in a real
and one-dimensional context. Thanks to Proposition \ref{prop:st}
below, the germs of a strongly transcendental holomorphic function
satisfy very few relations. Proposition \ref{exist:st} shows that
the definition is not empty: strongly transcendental holomorphic functions
do exist.

We will use the notation introduced in \ref{notation: tuples} and \ref{def: jet}.
\begin{defn}
\label{def:st} A holomorphic function $f:U\subseteq\mathbb{C}^{m}\to\mathbb{C}$
is \textbf{strongly transcendental }if, for every $(k,n)\in\mathbb{N}^{2}$
and for every $z=(z_{1}^{\left(m\right)},\dots,z_{n}^{\left(m\right)})$
an $n$-tuple of distinct points of $U$, we have
\[
\text{trdeg}_{\mathbb{Q}}\mathbb{Q}(z,\overline{z},j_{n}^{k}f(z), \overline{j_{n}^{k}f(z)})\ge\text{length}(j_{n}^{k}f(z),\bar{j_{n}^{k}f(z)}).
\]

\end{defn}
We first show that such functions exist, then we prove an independence
result for their germs at distinct points.
\begin{prop}
\label{exist:st} Let $U$ be an open subset of $\mathbb{C}^{m}$.
Then the set
\[
ST(U)=\{f\in\mathcal{O}(U);\; f\text{ is strongly transcendental }\}
\]
is residual in $\mathcal{O}(U)$ with respect to the topology induced
by uniform convergence on the compact subsets of $U$. In particular,
by the Baire Category Theorem, $ST(U)\neq\emptyset$.
\end{prop}
Note that strong transcendence is preserved under restriction but
not under gluing, so $ST$ is a presheaf but not a sheaf.
\begin{notation}
If $F:U\to V$ is a differentiable map between $C^{1}$ manifolds,
the differential of $F$ at $u\in U$ is denoted by $DF(u)$; it belongs
to the space of linear maps $\text{L}(T_{u}U,T_{F(u)}V)$, where $T_{u}U$
denotes the tangent space of $U$ at $u$.\end{notation}
\begin{proof}
The proof goes as follows. First we express $ST(U)$ as a countable
intersection of subsets $B_{n,k,P}$ of $\mathcal{O}(U)$. Each $B_{n,k,P}$
is the complement of the image of a certain map $\pi$. We apply a
version of the Sard-Smale Theorem for Fréchet manifolds, due to Eftekharinasab
in \cite{eft:sard}, to this $\pi$, and obtain that the set of regular
values of $\pi$ is residual. Moreover, we observe that $\pi$ is everywhere critical, so the set of its
regular values is the complement of its image, then coincides
with $B_{n,k,P}$. Hence, $ST(U)$ is a countable intersection of
residual sets, therefore also a residual set.

Let us recall that a point $(p,q)\in(\mathbb{C}^{2mn}\times\mathbb{C}^{2N})$
has transcendence degree over $\mathbb{Q}$ at least $2N$ if and only
if $(p,q)$ does not satisfy any nonsingular system of $2mn+1$ polynomial
equations over $\mathbb{Z}$. Hence, given an open subset $U$ of
$\mathbb{C}^{m}$, a function $f\in\mathcal{O}(U)$ belongs to $ST(U)$
if and only if

\[
\begin{array}{c}
{\displaystyle \forall n\in\mathbb{N}^{*},\;\forall k\in\mathbb{N},\;\forall P\in(\mathbb{Z}[y^{\left(2mn+2N\right)}])^{2mn+1},}\\
\forall z=(z_{1}^{\left(m\right)},\dots,z_{n}^{\left(m\right)})\in(\mathbb{C}^{m})^{n}\setminus\Delta_{n},\\
\mathrm{if}\ {\displaystyle \text{rk}(DP(z,\bar{z},j_{n}^{k}f(z),\bar{j_{n}^{k}f(z)}))=2mn+1}\\
\mathrm{then}\ P(z,\bar{z},j_{n}^{k}f(z),\bar{j_{n}^{k}f(z)})\neq0,
\end{array}
\]
where $N=n\left(\begin{array}{c}
m+k\\
k
\end{array}\right)$ is the cardinality of the tuple $j_{n}^{k}f(z)$ and $\Delta_{n}=\{(z_{1}^{\left(m\right)},\dots,z_{n}^{\left(m\right)})\in(\mathbb{C}^{m})^{n}:\;\exists i\neq j,\; z_{i}^{\left(m\right)}=z_{j}^{\left(m\right)}\}.$

In other words, $ST(U)$ can be described as the following countable
intersection

\[
ST(U)={\displaystyle \bigcap_{\begin{array}{c}
n\in\mathbb{N}^{*},\; k\in\mathbb{N},\\
P\in\mathbb{Z}[y^{\left(2mn+2N\right)}]
\end{array}}B_{n,k,P},}
\]

where
\[
\begin{array}{rl}
B_{n,k,P} & =\{f\in\mathcal{O}(U):\ \forall z=(z_{1}^{\left(m\right)},\dots,z_{n}^{\left(m\right)})\in(\mathbb{C}^{m})^{n}\setminus\Delta_{n}\\
 & \hspace{1cm}\text{rk}(DP(z,\bar{z},j_{n}^{k}f(z),\bar{j_{n}^{k}f(z)}))=2mn+1\Rightarrow P(z,\bar{z},j_{n}^{k}f(z),\bar{j_{n}^{k}f(z)})\neq0\}.
\end{array}
\]

We fix $n,k,P$ as above, define the set $V$ :
\[
\begin{array}{rl}
V:=\{(z,f)\in((\mathbb{C}^{m})^{n}\setminus\Delta_{n})\times\mathcal{O}(U): & \; P(z,\bar{z},j_{n}^{k}f(z),\bar{j_{n}^{k}f(z)})=0,\;\\
 & \text{rk}(DP(z,\bar{z},j_{n}^{k}f(z),\bar{j_{n}^{k}f(z)}))=2mn+1\},
\end{array}
\]
and equip $V$ with the topology induced by the family of seminorms,
indexed by the compact subsets $K$ of $U$:
\[
||\cdot||_{K}:\left(z,f\right)\mapsto\max\left\{ ||z||,\sup_{x\in K}|f(x)|\right\} .
\]
We can see $B_{n,k,P}$ as $\mathcal{O}(U)\setminus\pi\left(V\right)$,
where $\pi:V\rightarrow\mathcal{O}\left(U\right)$ is the restriction
to $V$ of the projection $\left(\mathbb{C}^{m}\right)^{n}\times\mathcal{O}\left(U\right)\to\mathcal{O}(U)$
onto the second factor.

Since the equations which define $V$ involve conjugations, $V$ does not inherit a complex structure. However, 
$V$ is a real Frechet manifold, whose tangent space $T_{\left(z,f\right)}V$ at $(z,f)\in V$
is the real vector space given by
\[
\begin{array}{rl}
{\displaystyle T_{\left(z,f\right)}V} & =\{(x,g)\in(\mathbb{C}^{m})^{n}\times\mathcal{O}(U):\;\\
 & \hspace{1cm}DP\left(z,\bar{z},j_{n}^{k}f\left(z\right),\bar{j_{n}^{k}f\left(z\right)}\right)
 \cdot
 \left[x,\bar{x},
 Dj_{n}^{k}f\left(z\right)\cdot x+j_{n}^{k}g\left(z\right),\bar{Dj_{n}^{k}f\left(z\right)\cdot x+j_{n}^{k}g\left(z\right)}\right]=0\}.
\end{array}
\]

Let us show that the Sard-Smale Theorem for Fréchet manifolds applies
to $\pi$. We fix $(z,f)\in V$, and denote by $\mathcal{M}_{z}^{k}$
the subset of $\mathcal{O}(U)$ of all the functions whose jets of
order $k$ at $z$ are equal to zero:
\[
\mathcal{M}_{z}^{k}=\{g\in\mathcal{O}(U):\; j_{n}^{k}g(z)=0\}.
\]
First observe that $\pi$ is a $C^{1}$-Lipschitz-Fredholm function:
\begin{itemize}
\item $\pi$ is a $C^{1}$ function, and
\[
D\pi(z,f):(x,g)\in T_{\left(z,f\right)}V\mapsto g\in\mathcal{O}(U)
\]
 is a Lipschitz (with constant 1) operator;
\item The kernel of $D\pi(z,f)$ is included in $(\mathbb{C}^{m})^{n}\times\{0\}$,
hence it has finite dimension;
\item Since $(0,g)\in T_{\left(z,f\right)}V$ if $g\in\mathcal{M}_{z}^{k}$,
we have
\[
\mathcal{M}_{z}^{k}\subseteq\text{Im}(D\pi(z,f)),
\]
so the co-kernel of $D\pi(z,f)$ is a quotient of $\mathcal{O}(U)/\mathcal{M}_{z}^{k}$.
The space $\mathcal{O}(U)/\mathcal{M}_{z}^{k}$ has finite dimension
$2N$ over $\mathbb R$, so co-ker$(D\pi(z,f))$ has finite dimension too.
\end{itemize}

It remains to compute the index of $D\pi(z,f)$. For this, observe
that
\begin{itemize}
\item $\text{ker}\, D\pi(z,f)\subseteq(\mathbb{C}^{m})^{n}\times\{0\}$.
In particular,
\[
\text{ker}\, D\pi(z,f)\cap\left((\mathbb{C}^{m})^{n}\times\mathcal{M}_{z}^{k}\right)\subseteq(\mathbb{C}^{m})^{n}\times\{0\};
\]

\item the equations which define $T_{\left(z,f\right)}V$ only involve (in
terms of $g$) the jet $j_{n}^{k}g(z)$ of order $k$.
\end{itemize}

Let $\widetilde{T_{\left(z,f\right)}V}$ be the quotient of $T_{\left(z,f\right)}V$
by $(\{0\}\times\mathcal{M}_{z}^{k})\cap T_{\left(z,f\right)}V$ and
$\widetilde{\mathcal{O}(U)}=\mathcal{O}(U)/\mathcal{M}_{z}^{k}$.
It follows from the previous observations that $D\pi(z,f)$ factors to
\[
\widetilde{D\pi(z,f)}:\widetilde{T_{\left(z,f\right)}V}\rightarrow\widetilde{\mathcal{O}(U)}
\]
and that the index of $D\pi(z,f)$ is equal to the index of $\widetilde{D\pi(z,f)}$.
Since $\widetilde{D\pi(z,f)}$ is a linear map between the finite
dimensional spaces $\widetilde{T_{\left(z,f\right)}V}$ and $\widetilde{\mathcal{O}(U)}$,
the index of $\widetilde{D\pi(z,f)}$ is simply $\text{dim}\,\widetilde{T_{\left(z,f\right)}V}\,-\,\text{dim}\,\widetilde{\mathcal{O}(U)}$.
Now,

\[
\begin{array}{rl}
{\displaystyle \widetilde{T_{\left(z,f\right)}V}} & =\{(x,g)\in(\mathbb{C}^{m})^{n}\times\widetilde{\mathcal{O}(U)}:\;\\
 & \hspace{1cm}DP\left(z,\bar{z},j_{n}^{k}f\left(z\right),\bar{j_{n}^{k}f\left(z\right)}\right)
 \cdot
 \left[x,\bar{x},
 Dj_{n}^{k}f\left(z\right)\cdot x+j_{n}^{k}g\left(z\right),\bar{Dj_{n}^{k}f\left(z\right)\cdot x+j_{n}^{k}g\left(z\right)}\right]=0\}.
\end{array}
\]
is a subspace of a real $\left(2mn+2N\right)$-dimensional space given by
$2(2mn+1)$ equations. Among this equations, at least $2mn+1$ are independent, since $\text{rk}\left(DP\left(z,\bar{z},j_{n}^{k}f\left(z\right),\bar{j_{n}^{k}f\left(z\right)}\right)\right)=2mn+1$.
Hence $\text{dim}(\widetilde{T_{\left(z,f\right)}})\le 2N-1$. On the
other hand, $\text{dim}\ \widetilde{\mathcal{O}(U)}=2N$, therefore
\[
\text{index}(D\pi(z,f))\le (2N-1)-2N=-1.
\]

We can now conclude. The version of the Sard-Smale Theorem in \cite[Theorem 4.3]{eft:sard}
applies to $\pi$ since $\pi$ is $C^{1}$ and $1>0=\max\left\{ 0,\text{index}\left(D\pi\left(z,f\right)\right)\right\} $.
We deduce that the set of regular values of $\pi$ is residual in
$\mathcal{O}(U)$. On the other hand, since it has negative index,
$D\pi(z,f)$ has non trivial co-kernel, so $\pi$ is nowhere a submersion.
In particular the set of its critical values coincides with its image.
Hence $B_{n,k,P}$, the complement of this image, is residual. Being
a countable intersection of residual sets, $ST(U)$ is residual, which
completes the proof.

\end{proof}
The following proposition expresses strong transcendence in terms
of the lack of relations between the germs at distinct points: except for the trivial operator, no operator in $\mathcal{D}^{*}$
vanishes at any tuple of germs of any strongly transcendental function.
\begin{prop}
\label{prop:st} Let $U$ be an open subset of $\mathbb{C}^{m}$,
$f\in ST(U)$ be a strongly transcendental holomorphic function on
$U$, $a=(a_{1},\dots,a_{k})\in(\mathbb{C}^{m})^{k}$ be a $k$-tuple
of distinct points of ${U}$, $b\in\mathbb{C}^{n}$ and $\mathcal{L}:\mathcal{O}_{a_{1}}\times\dots\times\mathcal{O}_{a_{k}}\to\mathcal{O}_{b}$
be an operator in $\mathcal{D}^{*}$. If $\mathcal{L}(f_{a_{1}},\dots,f_{a_{k}})=0$,
then there exists a neighbourhood $W$ of $(f_{a_{1}},\dots,f_{a_{k}})$
for the Krull topology such that $\mathcal{L}_{\upharpoonright W}=0$. \end{prop}
\begin{proof}
We follow the main steps of the proof of Lemma 3.6 in \cite{lg:generic}.
Suppose that $\mathcal{L}$ satisfies the hypotheses of the proposition.
Apply Proposition \ref{formalop} to $\mathcal{L}$ and $f_{i}=f_{a_{i}}$
to obtain, for all $(g_{1},\dots,g_{k})$ in some neighbourhood $W'$
of $(f_{a_{1}},\dots,f_{a_{k}})$,
\[
{\displaystyle \mathcal{L}(g_{1},\dots,g_{k})(x)=\sum_{\alpha\in\mathbb{N}^{n}}P_{\alpha}(c_{1},\dots,c_{\ell},j^{n_{\alpha}}g_{1}(a_{1}),\dots,j^{n_{\alpha}}g_{k}(a_{k}))(x-b)^{\alpha}}.
\]
From the fact that $\mathcal{L}(f_{a_{1}},\dots,f_{a_{k}})=0$, we
deduce that
\[
\forall\alpha\in\mathbb{N}^{n},\; P_{\alpha}(c_{1},\dots,c_{\ell},j_{k}^{n_{\alpha}}f(a_{1},\dots,a_{k}))=0.
\]
Denote by $c\in\mathbb C^{\ell}$ the tuple $c=(c_{1},\dots,c_{\ell})$.
Since $f\in ST(U)$, for all $\alpha\in\mathbb{N}^{n}$ the tuple
$(a,\bar{a},j_{k}^{n_{\alpha}}f(a),\bar{j_{k}^{n_{\alpha}}f(a)})$ satisfies
at most $2mk$ algebraically independent relations; so $(c,\bar{c},j_{k}^{n_{\alpha}}f(a),\bar{j_{k}^{n_{\alpha}}f(a)})$
satisfies at most $2\ell+2mk$ independent algebraic relations. In particular,
all but finitely many of the relations $P_{\alpha}(c,j_{k}^{n_{\alpha}}f(a))=0$
are dependent (recall that the $P_{\alpha}$ are polynomials over $\mathbb Q$ in their arguments and their conjugates). This implies that $\{n_{\alpha}:\alpha \in \mathbb N^n\}$ is bounded
by some constant $K\in\mathbb{N}$.

Let $W$ be the neighbourhood of $(f_{a_{1}},\dots,f_{a_{k}})$ given
by
\[
W= W' \cap \{(g_{1},\dots,g_{k})\in\mathcal{O}_{a_{1}}\times\dots\times\mathcal{O}_{a_{k}}:\;\forall i=1,\dots,k,\; j^{K}g_{i}(a_{i})=j^{K}f(a_{i})\}
\]
Then,
for every $(g_{1},\dots,g_{k})\in W$ and every $\alpha$, we have
$j^{n_{\alpha}}g_{i}(a_{i})=j^{n_{\alpha}}f(a_{i})$. It follows that
\[
\begin{array}{rcl}
{\displaystyle \mathcal{L}(g_{1},\dots,g_{k})(x)} & = & {\displaystyle \sum_{\alpha\in\mathbb{N}^{n}}P_{\alpha}(c_{1},\dots,c_{\ell},j^{n_{\alpha}}g_{1}(a_{1}),\dots,j^{n_{\alpha}}g_{k}(a_{k}))(x-b)^{\alpha}}\\
 & = & {\displaystyle \sum_{\alpha\in\mathbb{N}^{n}}P_{\alpha}(c_{1},\dots,c_{\ell},j_{k}^{n_{\alpha}}f(a_{1},\dots,a_{k}))(x-b)^{\alpha}}\\
 & = & \mathcal{L}(f_{a_{1}},\dots,f_{a_{k}})(x)\\
 & = & 0.
\end{array}
\]
Therefore $\mathcal{L}_{\upharpoonright W}=0$, which finishes the
proof.
\end{proof}

\section{Monomial division\label{sec:Monomial-division}}

In this section we let $\mathcal{A}=\{\exp\}$, and define
\[
f(z)=\frac{e^{z}-1}{z}\text{ for }z\neq0\text{ and }f(0)=1,
\]
so that $f$ is holomorphic and locally $\emptyset$-definable in
$\mathbb{R}_{\mathcal{A}\restri}$.

We shall prove Theorem A. This amounts to proving that the germ of
$f$ at zero belongs to $\mathcal{C}^{\emptyset}\setminus\mathcal{B}$
(so in particular $\mathcal{B}^{\emptyset}\subsetneq\mathcal{C}^{\emptyset}$
and $\mathcal{B}\subsetneq\mathcal{C}$).\\
{The strategy is the following. Using the differential equation satisfied by the exponential function, it is easy to see that the germs in $\mathcal{B}^{\emptyset}$ satisfy certain nonsingular systems of exponential polynomial equations (Lemma \ref{lem:implicit}). Therefore, if $f\in\mathcal{B}^{\emptyset}$, then there is a tuple $\Psi$ of analytic germs (and $f$ is one of them) such that the germs and their exponentials satisfy a nonsingular system of polynomial equations. This gives a certain upper bound $M$ on the transcendence degree of the tuple $(\Psi,\exp(\Psi))$ over $\mathbb{C}$. On the other hand, by Ax's Theorem \cite[Corollary 2]{ax:schanuel}, the components of the tuple $(\Psi,\exp(\Psi))$ satisfy few algebraic relations, so that the transcendence degree of the tuple must be at least $M+1$. This contradicts the fact that $f\in\mathcal{B}^{\emptyset}$ and proves Theorem A.}
\begin{notation}
\label{not: calcul diff}Let $k,l\in\mathbb{N}$, $x=\left(x_{1},\ldots,x_{k}\right)$
and $x'$ be a sub-tuple of $x$. If $F\left(x\right)=\left(F_{1}\left(x\right),\ldots,F_{l}\left(x\right)\right)$
is an $l$-tuple of functions and $i\in\left\{ 1,\ldots,l\right\} $,
then we denote by $\frac{\partial F_{i}}{\partial x'}\left(x\right)$
the column vector whose entries are $\frac{\partial F_{i}}{\partial x_{j}}\left(x\right)$,
for $x_{j}$ belonging to $x'$. We denote by $\frac{\partial F}{\partial x'}\left(x\right)$
the matrix whose columns are the vectors $\frac{\partial F_{i}}{\partial x'}\left(x\right)$,
for $i\in\left\{ 1,\ldots,l\right\} $. Finally, we denote by $e^{x}$
the $k$-tuple $\left(e^{x_{1}},\ldots,e^{x_{k}}\right)$.\end{notation}
\begin{defn}
Let $n\in\mathbb{N},\ x=\left(x_{0},x_{1},\ldots,x_{n}\right)$ and
$x'=\left(x_{1},\ldots,x_{n}\right)$. Let $a\in\mathbb{C}$, $g\in\mathcal{O}_{a}$
and $\mathcal{G}$ a sheaf of functions. We say that $g$ is \textbf{$n$-implicitly
defined from $\mathcal{G}$} if there exist $F=\left(F_{1},\ldots,F_{n}\right)\in\left(\mathcal{G}_{n+1}\right)^{n}$
and $\Psi=\left(\psi_{0},\psi_{1},\ldots,\psi_{n}\right)\in\left(\mathcal{O}_{a}\right)^{n+1}$
such that $\psi_{0}\left(z\right)=z,\ \psi_{1}\left(z\right)=g\left(z\right),\ F\left(\Psi\left(z\right)\right)=0$
and the matrix $\frac{\partial F}{\partial x'}\left(\Psi\left(z\right)\right)$
is invertible. In this case, $F$ is called an \textbf{implicit system}
(with coordinates in $\mathcal{G}$) of \textbf{size} $n$ and $\Psi$
is an \textbf{implicit solution}.\end{defn}
\begin{lem}\label{lem:implicit}
\label{a'} Let $a\in\mathbb{C}$ and $g\in\mathcal{O}_{a}$. Then
$g\in\mathcal{B}_{a}$ if and only if there exists $n\in\mathbb{N}$
such that $g$ is $n$-implicitly defined from $\mathbb{C}\left[x,e^{x}\right]$. \end{lem}
{The proof of the lemma is routine and is postponed to the end of the section.}
We now prove Theorem A.
\begin{proof}
[Proof of Theorem A] The germ $f_{0}$ of $f$ at zero belongs to
$\mathcal{C}^{\emptyset}$ since it is obtained from the germ $\exp_{0}\in\mathcal{A}_{0}$
by monomial division.

Suppose for a contradiction that $f_{0}\in\mathcal{B}_{0}$.

By Lemma \ref{a'}, there exists $n\in\mathbb{N}$ such that $f_{0}$
is $n$-implicitly defined from $\mathbb{C}\left[x,e^{x}\right]$.
Choose $n$ minimal with respect to the following property:
\[
\exists d\in\mathbb{N}^{\times}\ \text{s.t.}\ \frac{1}{d}f_{0}\left(dz\right)\text{\ is\ }n\text{-implicitly\ defined\ from\ }\mathbb{C}\left[x,e^{x}\right].
\]
In other words, for all $n',d\in\mathbb{N}^{\times}$, if ${\displaystyle \frac{1}{d}f_{0}}\left(dz\right)$
is $n'$-implicitly defined, then $n'\geq n$.

Let $x=\left(x_{0},x_{1},\ldots,x_{n}\right)$ and $x'=\left(x_{1},\ldots,x_{n}\right)$.
Let $F=\left(F_{1},\ldots,F_{n}\right)\in\left(\mathbb{C}\left[x,e^{x}\right]\right)^{n}$
and $\Psi=\left(\psi_{0},\psi_{1},\ldots,\psi_{n}\right)\in\left(\mathcal{O}_{a}\right)^{n+1}$
be such that $\psi_{0}\left(z\right)=z,\ \psi_{1}\left(z\right)=f_{0}\left(z\right),\ F\left(\Psi\left(z\right)\right)=0$
and the matrix ${\displaystyle \frac{\partial F}{\partial x'}}\left(\Psi\left(z\right)\right)$
is invertible. Write $F\left(x\right)=P\left(x,e^{x}\right)$, where
$P=\left(P_{1},\ldots,P_{n}\right)$ is an $n$-tuple of polynomials
in $2\left(n+1\right)$ variables, of total degree $\leq N$, for
some $N\in\mathbb{N}$.

Our first task is to use the minimality of $n$ to prove that the
components of the vector
\[
\left(z,f_{0}\left(z\right)-f_{0}\left(0\right),\psi_{2}\left(z\right)-\psi_{2}\left(0\right),\ldots,\psi_{n}\left(z\right)-\psi_{n}\left(0\right)\right)
\]
are $\mathbb{Q}$-linearly independent.

If this is not the case, then there are $d, a_{0},\ldots,a_{n-1}\in\mathbb{Z}$ where we may suppose without loss of generality that $d\in\mathbb{N}^{\times}$, and there is $K = \psi_n(0) - \sum_{i=0}^{n-1}\psi_i(0) \in\mathbb{C}$
such that
\[
d\psi_{n}\left(z\right)=\sum_{i=0}^{n-1}a_{i}\psi_{i}\left(z\right)+K.
\]
Let
\[
\begin{aligned}\tilde{x} & =\left(x_{0},x_{1},\ldots,x_{n-1}\right),\ \tilde{x}'=\left(x_{1},\ldots,x_{n-1}\right)\\
\varphi\left(\tilde{x}\right) & =\left(dx_{0},dx_{1},\ldots,dx_{n-1},\sum_{i=0}^{n-1}a_{i}x_{i}+K\right)\\
\eta\left(\tilde{x}\right) & =N\sum_{i=0}^{n-1}\left|a_{i}\right|x_{i}+K\\
\tilde{\Psi}\left(z\right) & =\left(\frac{\psi_{0}\left(z\right)}{d},\frac{\psi_{1}\left(z\right)}{d},\ldots,\frac{\psi_{n-1}\left(z\right)}{d}\right).
\end{aligned}
\]
For $i=1,\ldots,n$ we let $G_{i}\left(\tilde{x}\right)=e^{\eta\left(\tilde{x}\right)}F_{i}\left(\varphi\left(\tilde{x}\right)\right)$.
Note that $G_{i}\in\mathbb{C}\left[\tilde{x},e^{\tilde{x}}\right]$
and $G_{i}\left(\tilde{\Psi}\left(z\right)\right)=0$ for all $i\in\left\{ 1,\ldots,n\right\} $.

We claim that there are $i_{1},\ldots,i_{n-1}\in\left\{ 1,\ldots,n\right\} $
such that, if $\tilde{G}=\left(G_{i_{1}},\ldots,G_{i_{n-1}}\right)$,
then the matrix ${\displaystyle \frac{\partial\tilde{G}}{\partial\tilde{x}'}}\left(\tilde{\Psi}\left(z\right)\right)$
is invertible.

To see this, notice that for all $i=1,\ldots,n,\ j=1,\ldots,n-1$
\[
\frac{\partial G_{i}}{\partial x_{j}}\left(\tilde{x}\right)=e^{\eta\left(\tilde{x}\right)}\left[N\left|a_{j}\right|F_{i}\left(\varphi\left(\tilde{x}\right)\right)+d\frac{\partial F_{i}}{\partial x_{j}}\left(\varphi\left(\tilde{x}\right)\right)+a_{j}\frac{\partial F_{i}}{\partial x_{n}}\left(\varphi\left(\tilde{x}\right)\right)\right].
\]
Hence, if we let $G=\left(G_{1},\ldots,G_{n}\right)$, we have
\[
\frac{\partial G}{\partial x_{j}}\left(\tilde{\Psi}\left(z\right)\right)=e^{\eta\left(\tilde{\Psi}\left(z\right)\right)}\left[0+d\frac{\partial F}{\partial x_{j}}\left(\Psi\left(z\right)\right)+a_{j}\frac{\partial F}{\partial x_{n}}\left(\Psi\left(z\right)\right)\right].
\]
Using this last identity, it is easy to see that the matrix having
${\displaystyle \frac{\partial G}{\partial x_{1}}}\left(\tilde{\Psi}\left(z\right)\right),\ldots,{\displaystyle \frac{\partial G}{\partial x_{n-1}}}\left(\tilde{\Psi}\left(z\right)\right),{\displaystyle \frac{\partial F}{\partial x_{n}}}\left(\Psi\left(z\right)\right)$
as columns has rank $n$.

Hence there are $i_{1},\ldots,i_{n-1}\in\left\{ 1,\ldots,n\right\} $
such that the $\left(n-1\right)\times\left(n-1\right)$ minor of this
matrix relative to the rows $i_{1},\ldots,i_{n-1}$ and to the first
$n-1$ columns, has rank $n-1$. This proves the claim.

It follows that $\tilde{\Psi}\left(dz\right)$ satisfies the implicit
system $\tilde{G}=0$. Hence, ${\displaystyle \frac{1}{d}}f_{0}\left(dz\right)$
is $\left(n-1\right)$-implicitly defined, which contradicts the minimality
of $n$.

We have thus proved that the components of the vector
\[
\left(z,f_{0}\left(z\right)-f_{0}\left(0\right),\psi_{2}\left(z\right)-\psi_{2}\left(0\right),\ldots,\psi_{n}\left(z\right)-\psi_{n}\left(0\right)\right)
\]
are $\mathbb{Q}$-linearly independent.

By Ax's Theorem \cite[Corollary 2]{ax:schanuel}, we have that
\[
\text{trdeg}_{\mathbb{C}}\mathbb{C}\left(\Psi\left(z\right),e^{\Psi\left(z\right)}\right)\geq\left(n+1\right)+1=n+2.
\]
On the other hand, we claim that the tuple $\left(\Psi\left(z\right),e^{\Psi\left(z\right)}\right)$
satisfies a nonsingular system of $\left(n+1\right)$ polynomial equations
in $2\left(n+1\right)$ variables, and hence
\[
\text{trdeg}_{\mathbb{C}}\mathbb{C}\left(\Psi\left(z\right),e^{\Psi\left(z\right)}\right)\leq2\left(n+1\right)-\left(n+1\right)=n+1.
\]
This will provide a contradiction and will prove the theorem.

To prove this second claim, let $y=\left(y_{0},y_{1},\ldots,y_{n}\right)=\left(y_{0},y'\right)$
and $P_{0}\left(x,y\right)=x_{0}x_{1}+1-y_{0}$. Notice that, if $F_{0}\left(x\right)=P_{0}\left(x,e^{x}\right)$,
then $F_{0}\left(\Psi\left(z\right)\right)=0$.

We know that
\[
\frac{\partial F}{\partial x_{j}}\left(x\right)=\frac{\partial P}{\partial x_{j}}\left(x,e^{x}\right)+e^{x_{j}}\frac{\partial P}{\partial y_{j}}\left(x,e^{x}\right)\ \ \ \forall j=1,\ldots,n
\]
and that the matrix ${\displaystyle \frac{\partial F}{\partial x'}}\left(\Psi\left(0\right)\right)$
is invertible.

Hence there are $u_{1},\ldots,u_{n}\in\left\{ x_{1},\ldots,x_{n},y_{1},\ldots,y_{n}\right\} $
such that the vectors
\[
\frac{\partial P}{\partial u_{1}}\left(\Psi\left(0\right),e^{\Psi\left(0\right)}\right),\ldots,\frac{\partial P}{\partial u_{n}}\left(\Psi\left(0\right),e^{\Psi\left(0\right)}\right)
\]
are $\mathbb{C}$-linearly independent. Let $\tilde{P}=\left(P_{0},P_{1},\ldots,P_{n}\right)$.
Since the first coordinate of the vector ${\displaystyle \frac{\partial\tilde{P}}{\partial x_{0}}}\left(\Psi\left(0\right),e^{\Psi\left(0\right)}\right)$
is ${\displaystyle \frac{\partial P_{0}}{\partial x_{0}}}\left(\Psi\left(0\right),e^{\Psi\left(0\right)}\right)=f_{0}\left(0\right)=1$
and the first coordinate of the vectors
\[
{\displaystyle \frac{\partial\tilde{P}}{\partial u_{1}}}\left(\Psi\left(0\right),e^{\Psi\left(0\right)}\right),\ldots,{\displaystyle \frac{\partial\tilde{P}}{\partial u_{n}}}\left(\Psi\left(0\right),e^{\Psi\left(0\right)}\right)
\]
 is ${\displaystyle \frac{\partial P_{0}}{\partial u_{j}}}\left(\Psi\left(0\right),e^{\Psi\left(0\right)}\right)=0$,
these $\left(n+1\right)$ vectors are $\mathbb{C}$-linearly independent.
This proves the claim and finishes the proof of the theorem.
\end{proof}
{We finish this section by giving a proof of Lemma \ref{lem:implicit}.}
\begin{proof}
	We first proceed to prove, by induction on the size $n$ of the system,
	that $\mathcal{B}_{a}$ contains the coordinates of all implicit solutions
	of every implicit system with coordinates in $\mathcal{B}$. Since
	$\mathcal{A}\subseteq\mathbb{C}\left[x,e^{x}\right]\subseteq\mathcal{B}$,
	this will prove the right-to-left implication.
	
	If $n=1$, then the assertion follows from the fact that $\mathcal{B}_{a}$
	is closed under extracting implicit functions. If $n>1$, then suppose
	that $\mathcal{B}_{a}$ contains the coordinates of all implicit solutions
	of all implicit systems of size $n-1$ with coordinates in $\mathcal{B}$.
	
	Let $F$ be an implicit system of size $n$ with coordinates in $\mathcal{B}$
	and let $\Psi$ be an implicit solution. After possibly permuting
	the variables and the coordinates of $F$, we may suppose that $\frac{\partial F_{n}}{\partial x_{n}}\left(\Psi\left(a\right)\right)\not=0$. 
	
	Let $\tilde{x}=\left(x_{0},x_{1},\ldots,x_{n-1}\right)=\left(x_{0},\tilde{x}'\right),\ \tilde{\Psi}\left(z\right)=\left(\psi_{0}\left(z\right),\psi_{1}\left(z\right),\ldots,\psi_{n-1}\left(z\right)\right)$
	and let $\varphi\left(\tilde{x}\right)\in\mathcal{B}_{\tilde{\Psi}\left(a\right)}$
	be the implicit function of $F_{n}$ at $\tilde{\Psi}\left(a\right)$,
	so that $\psi_{n}\left(a\right)=\varphi\left(\tilde{\Psi}\left(a\right)\right)$.
	If we let $\tilde{F}_{i}\left(\tilde{x}\right)=F_{i}\left(\tilde{x},\varphi\left(\tilde{x}\right)\right)$
	for $i=1,\ldots,n$, then for $j=1,\ldots,n-1$, we have
	\[
	\frac{\partial\tilde{F}_{i}}{\partial x_{j}}\left(\tilde{x}\right)=\frac{\partial F_{i}}{\partial x_{j}}\left(\tilde{x},\varphi\left(\tilde{x}\right)\right)+\frac{\partial\varphi}{\partial x_{j}}\left(\tilde{x}\right)\frac{\partial F_{i}}{\partial x_{n}}\left(\tilde{x},\varphi\left(\tilde{x}\right)\right).
	\]
	Since the vectors ${\displaystyle \frac{\partial F}{\partial x_{1}}}\left(\Psi\left(z\right)\right),\ldots,{\displaystyle \frac{\partial F}{\partial x_{n}}}\left(\Psi\left(z\right)\right)$
	are linearly independent, there are $i_{1},\ldots,i_{n-1}\in\left\{ 1,\ldots,n\right\} $
	such that $\tilde{\Psi}$ is an implicit solution of the implicit
	system $\tilde{F}=\left(\tilde{F}_{i_{1}},\ldots,\tilde{F}_{i_{n-1}}\right)$,
	which has coordinates in $\mathcal{B}$, since $\mathcal{B}$ is closed
	under composition.
	
	By the inductive hypothesis, $\tilde{\Psi}\in(\mathcal{B}_{a})^{n-1}$,
	and, since
	\[
	\psi_{n}(z)=\varphi\left(\tilde{\Psi}\left(z\right)\right),
	\]
	we also have $\psi_{n}\in\mathcal{B}_{a}$.
	
	It remains to prove the left-to-right implication. Let $\mathcal{A}'$
	be the family of all holomorphic functions such that for every $k\in\mathbb{N}$
	and $a=a^{\left(k\right)}\in\mathbb{C}^{k}$, $\mathcal{A}'_{a}$
	is the collection of all germs $f\in\mathcal{O}_{a}$ such that $f$
	is $n$-implicitly defined from $\mathbb{C}\left[x,e^{x}\right]$
	(for some $n\in\mathbb{N}$). More precisely, let $z=z^{\left(k\right)}$
	and $t=t^{\left(n\right)}$ be tuples of variables and define
	
	\begin{align*}
		\mathcal{A}'_{a}= & \{f\in\mathcal{O}_{a}:\ \exists n\in\mathbb{N},\ \exists F\in\left(\mathbb{C}\left[z,t,e^{z},e^{t}\right]\right)^{k+n},\ \exists\Phi\in\left(\mathcal{O}_{a}\right)^{n-1}\ \text{such\ that},\\
		& \text{if\ }\Psi\left(z\right)=\left(z,f\left(z\right),\Phi\left(z\right)\right),\ \text{then}\ F\left(\Psi\left(z\right)\right)=0\ \text{and}\ \frac{\partial F}{\partial t}\left(\Psi\left(z\right)\right)\ \text{is\ invertible}\}.\tag{\ensuremath{*}}
	\end{align*}
	We prove that $\mathcal{A}'$ is stable under the operators in $\mathcal{B}^{*}$.
	Since $\mathcal{A}'\supseteq\mathcal{A}$ and $\mathcal{B}$ is the
	smallest collection containing $\mathcal{A}$ and stable under the
	action of $\mathcal{B}^{*}$, this will imply that $\mathcal{B}\subseteq\mathcal{A}'$.
	\begin{enumerate}
		\item \emph{Schwarz reflection}. \\
		It suffices to note that
		\[
		F\left(\Psi\left(z\right)\right)=0\ \text{and}\ \frac{\partial F}{\partial t}\left(\Psi\left(z\right)\right)\ \text{is\ invertible}\Leftrightarrow\overline{F}\left(\overline{\overline{\Psi\left(\overline{\overline{z}}\right)}}\right)=0\ \text{and}\ \frac{\partial\overline{F}}{\partial t}\left(\overline{\overline{\Psi\left(\overline{\overline{z}}\right)}}\right)\ \text{is\ invertible.}
		\]
		Hence, if $f\left(z\right)\in\mathcal{A}'_{a}$ then $g\left(z\right):=\overline{f\left(\overline{z}\right)}\in\mathcal{A}'_{\overline{a}}$.
		\item \emph{Composition}. \\
		Let $f\in\mathcal{A}'_{a}$ be $n$-implicitly defined (following
		the notation in $\left(*\right)$). Let $s\in\mathbb{N},\ b=b^{\left(s\right)}\in\mathbb{C}^{s}$
		and $g=\left(g_{1},\ldots,g_{k}\right)\in\left(\mathcal{A}'_{b}\right)^{k}$.
		Suppose that $g\left(b\right)=a$ and let $h:=f\circ g\in\mathcal{O}_{b}$.
		We aim to prove that $h\in\mathcal{A}'_{b}$.\\
		It is easy to see that we may suppose that the components of $g$
		are implicitly defined by the same implicit system $G$.
		In other words, we may suppose that there exists $m\in\mathbb{N}$
		such that, if $x=x^{\left(s\right)},\ u=u^{\left(m\right)}$ are tuples
		of variables, then there exist $G\in\left(\mathbb{C}\left[x,z,u,e^{x},e^{z},e^{u}\right]\right)^{k+m}$
		and $\Gamma\in\left(\mathcal{O}_{b}\right)^{m}$ such that, if $\Omega\left(x\right)=\left(x,g\left(x\right),\Gamma\left(x\right)\right)$,
		then $G\left(\Omega\left(x\right)\right)=0$ and $\frac{\partial G}{\partial\left(z,u\right)}\left(\Omega\left(x\right)\right)$
		is invertible.\\
		Let
		\[
		H\left(x,z,u,t\right):=\left(G\left(x,z,u\right),F\left(z,t\right)\right)\in\left(\mathbb{C}\left[x,z,u,t,e^{x},e^{z},e^{u},e^{t}\right]\right)^{k+m+n}
		\]
		and
		\[
		\Theta\left(x\right):=\left(\Omega\left(x\right),h\left(x\right),\Phi\left(g\left(x\right)\right)\right)\in\left(\mathcal{O}_{b}\right)^{s+k+m+n}.
		\]
		Then $H\left(\Theta\left(x\right)\right)=0$ and
		\[
		\frac{\partial H}{\partial\left(z,u,t\right)}\left(\Theta\left(x\right)\right)=\left(\begin{array}{cc}
		\frac{\partial G}{\partial\left(z,u\right)}\left(\Omega\left(x\right)\right) & 0\\
		\frac{\partial F}{\partial\left(z,u\right)}\left(\Psi\left(x\right)\right) & \frac{\partial F}{\partial t}\left(\Psi\left(x\right)\right)
		\end{array}\right),
		\]
		which is invertible. Hence, up to a permutation, $h$ is $\left(k+m+n\right)$-implicitly
		defined.
		\item \emph{Derivatives}. \\
		Let $f\in\mathcal{A}'_{a}$ be $n$-implicitly defined (following
		the notation in $\left(*\right)$) and let $i\in\left\{ 1,\ldots,k\right\} $.
		We aim to prove that $\frac{\partial f}{\partial z_{i}}\in\mathcal{A}'_{a}$.
		Let $w=w^{\left(n\right)}$ be a tuple of variables and
		\[
		\tilde{F}\left(z,t,w\right):=\frac{\partial F}{\partial z_{i}}\left(z,t\right)+\frac{\partial F}{\partial t}\left(z,t\right)\cdot w.
		\]
		Notice that $\frac{\partial\tilde{F}}{\partial w}\left(z,t,w\right)=\frac{\partial F}{\partial t}\left(z,t\right)$.
		Let
		\[
		F^{*}\left(z,t,w\right)=\left(F\left(z,t\right),\tilde{F}\left(z,t,w\right)\right)\in\left(\mathbb{C}\left[z,t,w,e^{z},e^{t},e^{w}\right]\right)^{2n}
		\]
		and
		\[
		\Psi^{*}\left(z\right)=\left(\Psi\left(z\right),\frac{\partial f}{\partial z_{i}}\left(z\right),\frac{\partial\Phi}{\partial z_{i}}\left(z\right)\right)\in\left(\mathcal{O}_{a}\right)^{k+2n}.
		\]
		Then $F^{*}\left(\Psi^{*}\left(z\right)\right)=0$ and
		\[
		\frac{\partial F^{*}}{\partial\left(t,w\right)}\left(\Psi^{*}\left(z\right)\right)=\left(\begin{array}{cc}
		\frac{\partial F}{\partial t}\left(\Psi\left(z\right)\right) & 0\\
		\frac{\partial\tilde{F}}{\partial t}\left(\Psi^{*}\left(z\right)\right) & \frac{\partial F}{\partial t}\left(\Psi\left(z\right)\right)
		\end{array}\right),
		\]
		which is invertible. Hence, up to a permutation, $\frac{\partial f}{\partial z_{i}}$
		is $2n$-implicitly defined.
		\item \emph{Implicit function}.\\
		Let $f\in\mathcal{A}'_{a}$ be $n$-implicitly defined (following
		the notation in $\left(*\right)$) and suppose that $\frac{\partial f}{\partial z_{k}}\left(a\right)\not=0$.
		Write $a=\left(a',a_{k}\right)\in\mathbb{C}^{k-1}\times\mathbb{C}$
		and $z=\left(z',z_{k}\right)$, where $z'=\left(z_{1},\ldots,z_{k-1}\right)$.
		Let $\varphi\in\mathcal{O}_{a'}$ be the implicit function of $f$
		at $a$. We aim to prove that $\varphi\in\mathcal{A}'_{a'}$. Write
		$t=\left(t_{1},t'\right)$, where $t'=\left(t_{2},\ldots,t_{n}\right)$,
		and let
		\[
		F^{*}\left(z',z_{k},t\right)=\left(F\left(z',z_{k},t\right),t_{1}\right)\in\left(\mathbb{C}\left[z',z_{k},t,e^{z'},e^{z_{k}},e^{t}\right]\right)^{n+1}
		\]
		and
		\[
		\Psi^{*}\left(z'\right)=\Psi\left(z',\varphi\left(z'\right)\right)\in\left(\mathcal{O}_{a'}\right)^{\left(k-1\right)+\left(n+1\right)}.
		\]
		Then $F^{*}\left(\Psi^{*}\left(z\right)\right)=0$ and
		\[
		\frac{\partial F^{*}}{\partial\left(z_{k},t_{1},t'\right)}(\Psi^{*}\left(z'\right))=\left(\begin{array}{ccc}
		\frac{\partial F}{\partial z_{k}} & \frac{\partial F}{\partial t_{1}} & \frac{\partial F}{\partial t'}\\
		0 & 1 & 0
		\end{array}\right)(\Psi\left(z',\varphi\left(z'\right)\right)).
		\]
		Derivating the identity $F\left(\Psi\left(z\right)\right)=0$ we deduce
		\[
		\frac{\partial F}{\partial t_{1}}\left(\Psi\left(z\right)\right)=-\left(\frac{\partial f}{\partial z_{k}}\left(z\right)\right)^{-1}\left[\frac{\partial F}{\partial z_{k}}\left(\Psi\left(z\right)\right)+\frac{\partial F}{\partial t'}\left(\Psi\left(z\right)\right)\cdot\frac{\partial\Phi}{\partial z_{k}}\left(z\right)\right],
		\]
		hence $\frac{\partial F}{\partial z_{k}}\left(\Psi^{*}\left(z'\right)\right)$
		is $\mathbb{C}$-linearly independent from the ($\mathbb{C}$-linearly
		independent) vectors
		\[
		\frac{\partial F}{\partial t_{2}}\left(\Psi^{*}\left(z'\right)\right),\ldots,\frac{\partial F}{\partial t_{n}}\left(\Psi^{*}\left(z'\right)\right).
		\]
		Therefore, $\varphi$ is $\left(n+1\right)$-implicitly defined.
	\end{enumerate}
\end{proof}

\section{Deramification\label{sec:Composition-with-roots}}

In this section we prove Theorem B.

In order to do this, we associate to each operator $\mathcal{L}$
in $\mathcal{D}^{*}$ a function which measures the index shift between
a (tuple of) series in the domain of $\mathcal{L}$ and the image
under $\mathcal{L}$ of such series.
\begin{defn}
\label{def: shift}Let $\mathcal{L}:\mathcal{O}_{a_{1}}\times\dots\times\mathcal{O}_{a_{m}}\to\mathcal{O}_{b}$
be an operator in $\mathcal{D}^{*}$, and for every $k\in\mathbb{N}$,
define
\[
B_{k}=\{(h_{1},\dots,h_{m})\in\mathcal{O}_{a_{1}}\times\dots\times\mathcal{O}_{a_{m}}:\;\forall i=1\dots,m,\; j^{k}h_{i}(a_{i})=0\}
\]
(note that $B_{k}$ is a neighbourhood of $0$ for the Krull topology).

The shift of $\mathcal{L}$ is the function $d_{\mathcal{L}}:\mathbb{N}\to\mathbb{N}$
given by
\begin{align*}
d_{\mathcal{L}}(n)= & \min\{k\in\mathbb{N}:\;\forall\left(f_{1},\ldots,f_{m}\right),\left(g_{1},\ldots,g_{m}\right)\in\mathcal{O}_{a_{1}}\times\dots\times\mathcal{O}_{a_{m}},\ \\
 & \text{if\ }\forall i=1\dots,m,\; j^{k}f_{i}\left(a_{i}\right)=j^{k}g_{i}\left(a_{i}\right),\ \text{then}\ j^{n}\mathcal{L}(f_{1},\ldots,f_{m})(b)=j^{n}\mathcal{L}(g_{1},\ldots,g_{m})(b)\}.
\end{align*}

\end{defn}
The shift function has the following interpretation: to compute the
terms of order $n$ of $\mathcal{L}(F)$ it suffices to consider the
set of terms of $F$ of order $k$ with $k\le d_{\mathcal{L}}(n)$.
Thanks to Proposition \ref{formalop}, the function $d_{\mathcal{L}}(n)$
is well defined for all $\mathcal{L}\in\mathcal{D}^{*}$. We can be
more precise: $d_{\mathcal{L}}(n)=n$ if $\mathcal{L}$ is either
the Schwarz reflection, or the composition, or the implicit function
operator; $d_{\mathcal{L}}(n)=n+1$ if $\mathcal{L}$ is a monomial
division or a partial derivative operator; finally, $d_{\mathcal{L}}(n)=mn$
if $\mathcal{L}$ is the $m$th-deramification operator. Moreover,
we have:
\begin{prop}
\label{shift:prop} Let $\mathcal{L}$ be an operator in $\mathcal{C}^{*}$.
Then, there exists a constant $N_{\mathcal{L}}\in\mathbb{N}$ such
that $d_{\mathcal{L}}(n)\le n+N_{\mathcal{L}}\ \forall n\in\mathbb{N}$. \end{prop}
\begin{proof}
The operators in $\mathcal{C}^{*}$ are finite compositions of operators
among polynomials, Schwarz reflections, compositions, implicit functions,
for which we set $N_{\mathcal{L}}=0$, and monomial division and partial
derivation, for which we set $N_{\mathcal{L}}=1$. To obtain a bound
on the shift for a general operator $\mathcal{L}\in\mathcal{C}^{*}$,
we first note that if $\mathcal{M},\mathcal{N}_{1},\ldots,\mathcal{N}_{k}$
are operators in $\mathcal{D}^{*}$, then the shift of the composition
satisfies
\[
d_{\mathcal{M\circ}(\mathcal{N}_{1},\ldots,\mathcal{N}_{k})}(n)\le\max_{i}d_{\mathcal{N}_{i}}(d_{\mathcal{M}}(n)).
\]
Hence, setting
\[
N_{\mathcal{M}(\mathcal{N}_{1},\ldots,\mathcal{N}_{k})}'=N_{\mathcal{M}}'+\max_{i}N_{\mathcal{N}_{i}}',
\]
we find integers $N_{\mathcal{L}}'$ associated to the representation
of $\mathcal{L}$ as a particular composition of elementary operators
in $\mathcal{C}^{*}$. We can then take $N_{\mathcal{L}}$ to be the
minimum of these integers over all representations of $\mathcal{L}$
as a composition of elementary operators in $\mathcal{C}^{*}$.
\end{proof}
This proposition shows in particular that no operator of composition
with roots belongs to $\mathcal{C}^{*}$. However, there are examples
of functions $f$ such that $z\mapsto f(\sqrt{z})$ is obtained from
$f$ via operators in $\mathcal{C}^{*}$. For instance, if $f(z)=\frac{1}{1+z^{2}}$,
then we have $f(\sqrt{z})=f(z)+\frac{1}{2}f'(z)$. Hence, to produce
our example we need a function which do not satisfy these kind of
relations. Strongly transcendental functions have this property.

The following proposition, together with Proposition \ref{prop:st},
implies Theorem B.
\begin{prop}
Let $f\in ST(U)$, where $U$ is an open neighbourhood of $0\in\mathbb{C}$.
Set $g(z)=f(z^{2})$ and $\mathcal{A}=\{g\}$. Then, the germ $f_{0}$
of $f$ at $0$ belongs to $\mathcal{D}_{0}^{\emptyset}\setminus\mathcal{C}_{0}$. \end{prop}
\begin{proof}
Since $g(z)=f(z^{2})$, the germ $f_{0}$ is the image of $g_{0}$
under the square deramification operator, and hence
belongs to $\mathcal{D}_{0}^{\emptyset}$.

Suppose now for a contradiction that $f_{0}\in\mathcal{C}_{0}$. Then
there exist distinct points $a_{0},a_{1},\dots,a_{k}\in\mathbb{C}$
and an operator $\mathcal{L}\in\mathcal{C}^{*}$ such that $f_{0}=\mathcal{L}(g_{0},g_{a_{1}},\dots,g_{a_{k}})$;
we may assume that $a_{0}=0$, so that the germ $g_{0}$ does indeed
occur in the list. Since $g(z)=f(z^{2})$, if we denote by $\mathcal{N}_{i}$
the operator $\mathcal{N}_{i}:\mathcal{O}_{{a_{i}^{2}}}\mapsto\mathcal{O}_{a_{i}}$
of composition with the polynomial $z\mapsto z^{2}$, we can rewrite
this equality as
\[
f_{0}=\mathcal{L}(\mathcal{N}_{0}(f_{0}),\mathcal{N}_{1}(f_{a_{1}^{2}}),\dots,\mathcal{N}_{k}(f_{a_{k}^{2}})).
\]
Hence the operator $\mathcal{M}:\left(\phi_{0},\dots,\phi_{k}\right)\mapsto\mathcal{L}(\mathcal{N}_{0}(\phi_{0}),\mathcal{N}_{1}(\phi_{1}),\dots,\mathcal{N}_{k}(\phi_{k}))-\phi_{0}$,
vanishes at the tuple of the germs at distinct points of a strongly
transcendental function. By Proposition \ref{prop:st},
since $\mathcal{M}\in\mathcal{C}^{*}$, this operator vanishes identically
on a neighbourhood of $(f_{0},f_{a_{1}^{2}},\ldots,f_{a_{n}^{2}})$.
For a large enough $\ell$, we have that $\mathcal{M}$ vanishes at
$(f_{0}+\lambda z^{\ell},f_{a_{1}^{2}}\dots,f_{a_{n}^{2}})$ 
for all $\lambda\in\mathbb{C}$, from which it follows that%
\[
\forall\lambda\in\mathbb{C},\; f_{0}+{\lambda}z_{0}^{\ell}=\mathcal{L}(g_{0}+\lambda z^{2\ell},g_{a_{1}},\dots,g_{a_{k}}).
\]
From this expression we deduce that $d_{\mathcal{L}}(\ell)\ge2\ell$
for all sufficiently large $\ell$. This, together with Proposition
\ref{shift:prop}, contradicts the hypothesis $\mathcal{L}\in\mathcal{C}^{*}$.
\end{proof}

\section{Blow-downs\label{sec:Blow-downs}}

In this section we prove Theorem C.
\begin{notation}
\label{notation: blow ups}We fix the usual coordinate system for
the blow-up of $0\in\mathbb{C}^{2}$.

Recall that the blow-up of $0\in\mathbb{C}^{2}$ is the map $\pi:V\to\mathbb{C}^{2}$
where
\[
V=\{(z,p)\in\mathbb{C}^{2}\times\mathbb{C}\mathbb{P}^{1}:\; z\in p\}
\]
and $\pi(z,p)=z.$ Let $D=\pi^{-1}(0)=\{0\}\times\mathbb{C}\mathbb{P}^{1}$
be the exceptional divisor. On the analytic manifold $V$ we consider
the atlas given by the following charts $c_{\lambda}$, for $\lambda\in\overline{\mathbb{C}}:=\mathbb{C}\cup\{\infty\}$:
\[
\begin{array}{l}
c_{\lambda}:\mathbb{C}^{2}\ni(z_{1},z_{2})\mapsto(z_{1},(\lambda+z_{2})z_{1},[1:\lambda+z_{2}])\text{ if }\lambda\neq\infty\\
c_{\infty}:\mathbb{C}^{2}\ni(z_{1},z_{2})\mapsto(z_{1}z_{2},z_{2},[z_{1}:1]),
\end{array}
\]
which, after composition with $\pi$, give rise to the following system
$\pi_{\lambda}=\pi\circ c_{\lambda}$ of local expressions for $\pi$:
\[
\begin{array}{l}
\pi_{\lambda}:\mathbb{C}^{2}\ni(z_{1},z_{2})\mapsto(z_{1},(\lambda+z_{2})z_{1})\text{ if }\lambda\neq\infty\\
\pi_{\infty}:\mathbb{C}^{2}\ni(z_{1},z_{2})\mapsto(z_{1}z_{2},z_{2})
\end{array}.
\]
If $f:U\subseteq\mathbb{C}^{2}\to\mathbb{C}$ is a function, then
the \textbf{blow-up} of $f$ is the function $f\circ\pi:\pi^{-1}(U)\to\mathbb{C}$,
and $f$ is the \textbf{blow-down }of $f\circ\pi$. Since we do not
want to introduce sheaves on manifolds other than $\mathbb{C}^{n}$,
we will only use local coordinates. Hence, the \textbf{blow-up of
$f:U\subseteq\mathbb{C}^{2}\to\mathbb{C}$ centered }at $\lambda$
is the function $f\circ\pi_{\lambda}:\pi_{\lambda}^{-1}(U)\to\mathbb{C}$;
we then say that $f$ is the \textbf{blow-down of the family $(f\circ\pi_{\lambda})_{\lambda\in\overline{\mathbb{C}}}$}.
Note that the blow-ups of $f$ are obtained from $f$ by composing
with polynomials (operators in $\mathcal{B}^{*}$). Hence, the fact
that $\mathcal{E}$ is stable under blow-downs implies that
\[
f_{0}\in\mathcal{E}_{0}\Leftrightarrow\forall\lambda\in\overline{\mathbb{C}},\;(f\circ\pi_{\lambda})_{0}\in\mathcal{E}_{0}.
\]

\end{notation}
The key idea which allows us to construct our counterexample is the
observation that blow-downs are not \emph{local} operators. Let $g:V\to\mathbb{C}$
be a function and suppose that the blow-down $f$ of $g$ exists.
Choosing $f$ wisely, we can show that to obtain the germ $f_{0}$,
we need the germs of $g$ at \emph{all} points of the exceptional
divisor $D$, whereas if $f$ were obtained from $g$ via operators
from $\mathcal{D}^{*}$, one would only need the germs of $g$ at
finitely many points to construct $f_{0}$. To make this argument
work we need to choose an $f$ which satisfies very few relations.
The following proposition is useful to establish that strongly transcendental
functions have this good behaviour.
\begin{prop}
\label{bd:prop} Let $f\in\mathcal{S}T(U)$ be strongly transcendental
on an open subset $U$ of $\mathbb{C}^{2}$, let $\mathcal{L}$ be an
operator in $\mathcal{D}^{*}$ of arity $n$  and let $(b_{0},b_{1},\dots,b_{n})\in U^{n+1}$
be distinct points in $U$%
. Then $f_{b_{0}}\neq\mathcal{L}(f_{b_{1}},\dots,f_{b_{n}})$. \end{prop}
\begin{proof}
Suppose $f_{b_{0}}=\mathcal{L}(f_{b_{1}},\dots,f_{b_{n}})$. Then
the operator $\mathcal{M}$ defined by
\[
\mathcal{M}(h_{0},h_{1},\dots,h_{n})=\mathcal{L}(h_{1},\dots,h_{n})-h_{0}
\]
belongs to $\mathcal{D}^{*}$ as the composition of the operator of
composition with a polynomial and the operator $\mathcal{L}$. Moreover,
$\mathcal{M}$ vanishes at $(f_{b_{0}},f_{b_{1}},\dots,f_{b_{n}})$,
and hence, according to Proposition \ref{prop:st}, on some neighbourhood
$W$ of $(f_{b_{0}},f_{b_{1}},\dots,f_{b_{n}})$.

For every $k\in\mathbb{N}$, let
\[
B_{k}=\{(h_{0},\dots,h_{n})\in\mathcal{O}_{b_{0}}\times\dots\times\mathcal{O}_{b_{n}}:\; j^{k}h_{0}(b_{0})=0,\dots,j^{k}h_{n}(b_{n})=0\}
\]
and choose $k\in\mathbb{N}$ sufficiently large such that $(f_{b_{0}},\dots,f_{b_{n}})+B_{k}\subseteq W$.
Denote by $b_{i}=(b_{1,i},b_{2,i})$ the coordinates of $b_{i}$ for
$i=0,\dots,n$, and let $P$ be the polynomial
\[
P(z_{1},z_{2})={\displaystyle {\displaystyle \prod_{i=0}^{n}}(z_{1}-b_{1,i})^{k}(z_{2}-b_{2,i})^{k}.}
\]

Since $(P_{b_{0}},P_{b_{1}},\dots,P_{b_{n+1}})$ and $(0,P_{b_{1}},\dots,P_{b_{n}})$
both belong to $B_{k}$, we have
\[
\mathcal{L}(f_{b_{1}}+P_{b_{1}},\dots,f_{b_{n}}+P_{b_{n}})-(f_{b_{0}}+P_{b_{0}})=0
\]
\[
\text{and }\mathcal{L}(f_{b_{1}}+P_{b_{1}},\dots,f_{b_{n}}+P_{b_{n}})-(f_{b_{0}})=0,
\]
from which we deduce $P_{b_{0}}=0$, which is absurd.
\end{proof}
The following proposition, together with Proposition \ref{prop:st},
implies Theorem C.
\begin{prop}
\label{ex3:prop} Let $U$ be a neighbourhood of $0\in\mathbb{C}^{2}$
and $f:U\to\mathbb{C}$ be a strongly transcendental function. Let
$\pi$ be the blow-up of $0\in\mathbb{C}^{2}$ and $\mathcal{A}=\{f\circ\pi_{\lambda}:\;{\lambda\in\overline{\mathbb{C}}}\}$.
Then $f_{0}\in\mathcal{E}_{0}^{\emptyset}\setminus\mathcal{D}_{0}$. \end{prop}
\begin{proof}
The germ $f_{0}$ is in $\mathcal{E}_{0}^{\emptyset}$ since it is
the blow-down of the family $\{(f\circ\pi_{\lambda})_{0}:\;{\lambda\in\overline{\mathbb{C}}}\}$,
whose elements belong to $\mathcal{A}_{0}$.

Let us suppose for a contradiction that $f_{0}\in\mathcal{D}_{0}$.
Then there exist $\mathcal{N}\in\mathcal{D}^{*}$, $(a_{1},\dots,a_{k})\in(\mathbb{C}^{2})^{k}$
and $(\lambda_{1},\dots,\lambda_{k})\in(\overline{\mathbb{C}})^{k}$
such that
\begin{equation}
f_{0}=\mathcal{N}((f\circ\pi_{\lambda_{1}})_{a_{1}},\dots,(f\circ\pi_{\lambda_{k}})_{a_{k}}).\label{eq:1}
\end{equation}

Let $U'$ be a neighbourhood of $\{c_{\lambda_{1}}(a_{1}),\dots,c_{\lambda_{k}}(a_{k})\}$
in $V$ such that $D\not\subseteq\overline{U'}$ (so that $\pi(U')$ is not a neighbourhood of $0$). Define $\mathcal{A}'=\{(f\circ\pi_{\lambda})_{\upharpoonright{c_{\lambda}^{-1}(U')}}:\;\lambda\in\overline{\mathbb{C}}\}$
and let $\mathcal{D}'$ be the closure of $\mathcal{A}'$ under the
action of $\mathcal{D}^{*}$.

Equation (\ref{eq:1}) above shows that $f_{0}\in\mathcal{D}'_{0}$,
so by definition there exists a neighbourhood $U''$ of $0$ such
that $f_{\upharpoonright U''}\in\mathcal{D}'(U'')$. Since $\pi(U')$ is
not a neighbourhood of $0$, $U''\setminus\pi(U')\neq\emptyset$. We fix $b_{0}\in U''\setminus\pi(U')$.
Now, since $f_{\upharpoonright U''}\in\mathcal{D}'(U'')$, the germ
of $f$ at $b_{0}$ can be expressed in terms of some of the germs
of the restriction of $f\circ\pi$ to $U'$. So there exist $\mathcal{M}\in\mathcal{D}^{*}$,
$(\mu_{1},\dots,\mu_{n})\in(\overline{\mathbb{C}})^{n}$ and $(a'_{1},\dots,a'_{n})\in(\mathbb{C}^{2})^{n}$
with $a'_{i}\in c_{\mu_{i}}^{-1}(U')$ for $i=1,\dots,n$, such that
\[
f_{b_{0}}=\mathcal{M}((f\circ{\pi_{\mu_{1}}})_{a'_{1}},\dots,(f\circ\pi_{\mu_{n}})_{a'_{n}}).
\]

This equation can be reformulated as a relation between the germs
of $f$ at $b_{0}$ and $b_{1},\dots,b_{n}$, where $b_{i}=\pi_{\mu_{i}}(a_{i}')$.
If $\mathcal{N}_{i}:\mathcal{O}_{b_{i}}\to\mathcal{O}_{a'_{i}}$ is
the operator of composition with the polynomial $\pi_{\mu_{i}}$,
then $\mathcal{N}_{i}\in\mathcal{B}^{*}$ and $(f\circ{\pi_{\mu_{i}}})_{a'_{i}}=\mathcal{N}_{i}(f_{b_{i}})$.
So we obtain
\[
f_{b_{0}}=\mathcal{M}(\mathcal{N}_{1}(f_{b_{1}}),\dots,\mathcal{N}_{n}(f_{b_{n}})).
\]
Setting $\mathcal{L}=\mathcal{M}(\mathcal{N}_{1},\dots,\mathcal{N}_{n})$,
the operator $\mathcal{L}$ is in $\mathcal{D}^{*}$, since it is
a composition of operators of $\mathcal{D}^{*}$, and we have $f_{b_{0}}=\mathcal{L}(f_{b_{1}},\dots,f_{b_{n}})$.
Up to decreasing $n$, we may suppose that the points $b_{1},\dots,b_{n}$
are all distinct. Moreover, observe that $b_{0}\notin\{b_{1},\dots,b_{n}\}$:
for $i=1,\dots,n$, $a'_{i}\in c_{\mu_{i}}^{-1}(U')$, so $b_{i}\in\pi(U')$
while $b_{0}$ has been chosen outside $\pi(U')$. Hence the points
$b_{0},\dots,b_{n}$ are all distinct, so Proposition \ref{bd:prop}
applies. But this contradicts the fact that $f_{b_{0}}=\mathcal{L}(f_{b_{1}},\dots,f_{b_{n}})$.
\end{proof}
\bibliographystyle{alpha}
\newcommand{\SortNoop}[1]{}

\end{document}